%% file: arxiv.tex
\documentclass[aos,preprint,leqno]{imsart}

\usepackage[numbers]{natbib}

\usepackage{multicol}

\usepackage[dvips]{graphics}
\DeclareGraphicsExtensions{.eps.gz,.eps,.epsi.gz,.epsi,.ps,.ps.gz}
\usepackage{epsfig}
\usepackage{color}
\graphicspath{{fig/}}

\usepackage[dvipsnames]{xcolor}

\usepackage{latexsym}
\usepackage{graphicx}
\usepackage{amsmath}
\usepackage{amsthm}
\usepackage{dsfont}
\usepackage{amsfonts}
\usepackage{amssymb}
\usepackage{mathtools}
\usepackage{mathrsfs}
\usepackage[lofdepth,lotdepth]{subfig}
\usepackage{multirow}
\usepackage{booktabs}
\usepackage{lscape}
\usepackage{constants}

\usepackage{hyperref}
\usepackage{cleveref}
\hypersetup{
    breaklinks=true,
    linkcolor=blue,
    citecolor=blue,
    colorlinks=true,
pdfborder={0 0 0}}

\usepackage{thmtools}
\usepackage{thm-restate}

\numberwithin{equation}{section}

\input{preambule-cun-hui.tex}

\declaretheorem[name=Theorem,numberwithin=section]{theorem}

\declaretheorem[name=Corollary,sibling=theorem]{corollary}

\crefname{assumption}{assumption}{assumptions}

\usepackage{marginnote}

\usepackage{soul}

\usepackage[textsize=scriptsize,textwidth=1.1in]{todonotes}

\usepackage[color]{changebar}
\cbcolor{Apricot}
\def\calB{{\mathscr B}}

\def\RR{\R}
\def\EE{\E}
\def\bias{\hbox{\rm bias}}

\def\scrH{\calH}

\begin{document}

\title{
High-Order Statistical Functional Expansion and
Its Application To Some Nonsmooth Problems
}
\runtitle{High-Order Statistical Functional Expansion}
\author{Fan Zhou$^{1}$, Ping Li$^{1}$ and Cun-Hui Zhang$^{2}$}
\affiliation{Baidu Research$^{1}$ and Rutgers University$^{2}$
\thanks{The work of Cun-Hui Zhang was conducted as a consulting researcher at Baidu Research -- 10900 NE 8th St. Bellevue, WA 98004, USA}
}
\date{\today}

\begin{abstract}
Let $\bx_j = \btheta +\bep_j, j=1,...,n$, be observations of an unknown parameter $\btheta$ in a Euclidean or separable Hilbert space $\scrH$, where $\bep_j$ are noises as random elements in $\scrH$ from a general distribution. We study the estimation of $f(\btheta)$ for a given functional $f:\scrH\rightarrow \RR$ based on $\bx_j$'s. The key element of our approach is a new method which we call High-Order Degenerate Statistical Expansion. It leverages the use of classical multivariate Taylor expansion and degenerate $U$-statistic and yields an elegant explicit formula. In the univariate case of $\scrH=\R$, the formula expresses the error of the proposed estimator as a sum of order $k$ degenerate $U$-products of the noises (with no tied index for each $k$-product) with coefficient $f^{(k)}(\btheta)/k!$ and an explicit remainder term in the form of the Riemann-Liouville integral as in the Taylor expansion around the true $\btheta$. For general $\scrH$, the formula expresses the estimation error in terms of the inner product of $f^{(k)}(\btheta)/k!$ and the average of the tensor products of $k$ noises with distinct indices and a parallel extension of the remainder term from the univariate case. This makes the proposed method a natural statistical version of the classical Taylor expansion. The proposed estimator can be viewed as a jackknife estimator of an ideal degenerate expansion of $f(\cdot)$ around the true $\btheta$ with the degenerate $U$-product of the noises, and can be approximated by bootstrap. Thus, the jackknife, bootstrap and Taylor expansion approaches all converge to the proposed estimator. We develop risk bounds for the proposed estimator under proper moment conditions and a central limit theorem under a second moment condition (even in expansions of higher than the second order).  We apply this new method to generalize several existing results with smooth and nonsmooth $f$ to universal $\bep_j$'s with only minimum moment constraints.
\end{abstract}

\maketitle

\section{Introduction}

We consider
\begin{equation}
	\label{model}
	\bx_j = \btheta +\bep_j,~j=1,...,n
\end{equation}
where $\btheta$ is an unknown parameter and belongs to a finite-dimensional Euclidean or a separable Hilbert space $\scrH$, and $\bep_j$ are mean zero random noises in $\scrH$ with general distributions.
The goal of this article is to study the estimation of $f(\btheta)$ for a given functional $f: \scrH \rightarrow \RR$ when the complexity parameter of the problem is large under only minimal moment constraints.

The study of the estimation of functionals of high-dimensional or infinite-dimensional parameter has a long history.
Notable results include but not limited to~\cite{levit1976efficiency,levit1978asymptotically,ibragimov1986some,bickel1988,nemirovskii1991necessary,birge1995estimation,laurent1996efficient,lepski1999estimation,nemirovski2000topics,ibragimov2013statistical}. Two types of special functionals are extensively studied, namely the linear and quadratic functionals.
Results on the estimation of linear functionals include~\cite{donoho1987minimax,donoho1991geometrizing,klemela2001sharp,cai2005adaptive} and the references therein.
Results on the estimation of quadratic functionals include~\cite{donoho1990minimax,laurent2000adaptive,bickel2003,cai2005nonquadratic,klemela2006sharp} and the references therein.
Recently there is a noticeable surge of interest
in efficient and minimax rate optimal estimation of functionals of parameter in high dimensional models or models with growing dimension, see~\cite{collier2017minimax,doi:10.1111/rssb.12026,van2014asymptotically,koltchinskii2021,zhou2019fourier}.
However, these papers are mostly focused on the case of Gaussian noise which is
quite restrictive in practice.  Thus it is of great interest to develop robust inferential procedures that are less sensitive to the distributional assumptions.
This is a major motivation for us to study
the functional estimation of high-dimensional parameters under low moment constraints on the noise distribution.

A straightforward approach to this problem is to use the plug-in estimator $f(\bar{\bx})$ with the sample mean $\bar{\bx}$ of the observations.
One may think of $f(\bar{\bx})$ as a good estimator as $\bar{\bx}$ is the maximum likelihood estimator of
$\btheta$ when the noise $\bep_j$ are Gaussian. However, interestingly, our problem is more delicate.
Even if in the Gaussian shift model,
there are situations where the plug-in estimator is sub-optimal
due to its large bias as soon as the complexity parameter characterized by the effective rank of $\bep$'s covariance operator $r(\bSigma)$ is of greater order than $n^{1/2}$.
Especially, a recent work
\cite{koltchinskii2021} has shown that when $\scrH = \RR^d$ under the Gaussian shift model with $r(\bSigma) = n^{\alpha}$, $\alpha\in (0,1]$, there are functionals $f$ with smoothness index $1< s<1/(1-\alpha)$ such that no $\sqrt{n}$-consistent
estimators of $f(\btheta)$ exists.
Thus, efficient estimation of $f(\btheta)$ under model (\ref{model}) poses a challenging problem.

\subsection{Related works}
Several methods were proposed recently to address the problem in the Gaussian shift model.
One of them, developed in a series of works \cite{koltchinskii2020asymptotically,koltchinskii2021,koltchinskii2019estimation},
is to use an iterative bootstrap technique to correct the bias of the plug-in estimator.
We briefly summarize the idea of iterative bootstrap as follows.
Denote by $\calT$ the linear operator given by
\begin{equation*}
	\calT g(\btheta):= \EE_{\btheta}g(\btheta) = \EE g(\btheta+ \bar{\bep}),
\end{equation*}
where $\bar{\bep}=n^{-1} \sum_{j=1}^n \bep_j$.
Let $\calI$ be the identity operator and $\calB := \calT-\calI$.
To create an estimator $g(\bar{\bx})$ of $f(\btheta)$ with small bias,
it is tempting to solve the integral equation
\bes
\calT g(\btheta) = (\calI + \calB) g(\btheta) = f(\btheta)
\ees
as accurately as possible.
However, solving such an integral equation itself is challenging. Instead, the authors used a natural finite approximation of the Neumann series $(\calI + \calB)^{-1} = (\calI -\calB +\calB^2 - \calB^3 +\dots)$ to create the following estimator
\begin{equation}
	\label{estimator_0}
	f_{k}(\bar{\bx}) := \sum_{j=0}^{k}(-1)^j\calB^j f(\bar{\bx})
\end{equation}
with
\begin{equation}
\label{bootstrap}
    \calB^j f(\bar{\bx}) := \EE_{\cdot |\bar{\bx}}\bigg[f^{(j)}\bigg(\bar{\bx}+\sum_{\ell=1}^j \tau_{\ell}\tilde{\bep}_{\ell}\bigg) (\tilde{\bep}_1,\dots,\tilde{\bep}_j)  \bigg]
\end{equation}
where $\EE_{\cdot |\bar{\bx}}$ is the conditional expectation taking over the i.i.d. random variables
$\{\tau_i\}_{i=1}^j \sim U[0,1]$ and i.i.d. bootstrap samples $\{\tilde{\bep}_i\}_{i=1}^j$ of $\bar{\bep}$.
It was shown that \eqref{bootstrap} holds exactly for $\calB=\calT-\calI$ so that
the bias of \eqref{estimator_0} is exactly $(-1)^{k}\calB^{k+1}f(\btheta)$. Intuitively, the main idea is that given
the condition that the bootstrap operator $\calB$ is small,
for relatively small noise $\bep$ and sufficiently smooth $f$,
the bias should also be small. Thus bias reduction can be achieved.
In practice, the expectation in \eqref{bootstrap} can be replaced by averaging samples obtained from Monte-Carlo simulations, of which performance can be guaranteed by law of large numbers \cite{zhou2021icml}.

{Another work \cite{jiao2020} compared multiple jackknife, iterative bootstrap and Taylor expansion approaches
in a study of the estimation of
$f(\theta)$ for a given $f\in C[0,1]$ based on Bernoulli$(\theta)$ observations.
Although \cite{jiao2020} studied the problem in a very specific classical model,
the methods they used to approach the problem are closely related to ours.
Especially, in Sec. A of their paper, they
proposed a multiple jackknife estimator
\begin{equation}
    \hat{f}_m := \sum_{k=1}^m C_k \frac{n_k!}{n!}\sum_{1\le i_1\neq \cdots \neq i_{n_k}\le n}
    f\bigg(\frac{1}{n_k}\sum_{j=1}^{n_k} x_{i_j}\bigg),
\end{equation}
with sample sizes $n_1<n_2<\dots<n_m \leq n$ and constants $C_k$ satisfying
$\sum_{k=1}^m C_k/n_k^\rho = I\{ \rho=0\}$ for $\rho=0,\ldots,m-1$.
This estimator makes use of $U$-statistics from sub-sampling data to cancel each other's bias
in order to achieve overall bias-reduction when the bias of $f(\bar{\bx})$ is a sum of
$\sum_{\rho=1}^{m-1} \bias_\rho/n^\rho$ and a higher order term. In Sec. C, they % authors
generalized the bias correction technique of first order Taylor expansion to higher orders in an iterative way.
However, as they noted in \cite{jiao2020}, an iterative bias correction into Taylor series of higher orders can be tedious and computationally intensive even in the one dimensional case. As a result, they proposed to use the sample splitting technique and considered the Taylor expansion of $f(\hat{\theta})$ at $\theta$:
\begin{equation}
    \hat{f}_m:=\sum_{k=0}^{2m-1} \frac{f^{(k)}(\hat{\theta}^{(1)})}{k!}
    \sum_{j=0}^k {k \choose j} \hat{\theta^j}^{(2)}(-\hat{\theta}^{(1)})^{k-j}
\end{equation}
where $\hat{\theta^j}^{(2)}$ is an unbiased estimator of $\theta^j$,
and $\hat{\theta}^{(1)}$ and $\hat{\theta^j}^{(2)}$ are obtained from split samples
independent of each other.}

In a series of works \cite{cai2011,collier2017minimax,collier2020,collier2019minimax},
the authors studied rate minimax estimation of $f(\btheta)$ under Gaussian shift model
when
$f(\btheta) = \|\btheta\|_p$
or $f(\btheta) = \|\btheta\|_p^p$.
The main idea of their works is to use Hermite polynomials to approximate the given functional,
possibly with some delicate truncation on the magnitude of each observational coordinate to address sparsity presented in $\btheta$.

Based on the seminal work on unbiased estimation \cite{kolmogorov1950unbiased} and the fruitful idea of Littlewood-Paley theory, \cite{zhou2019fourier} proposed a Fourier analytical estimator
\begin{equation}
	\label{estimator: fourier}
	\hat{f}_R(\bar{\bx}) :=   \frac{1}{(2\pi)^{d/2}}   \int_{\Omega}  \mathcal{F} f(\bzeta)e^{\langle \bSigma \bzeta, \bzeta \rangle/2n} e^{i\bzeta \cdot \bar{\bx}} d\bzeta,
\end{equation}
where $\Omega:= \{\bzeta: \|\bzeta\|\leq R \} \subset \RR^d$ is a properly truncated region in the frequency domain. The main idea is that the factor $e^{\langle \bSigma \bzeta, \bzeta \rangle/2n}$ in (\ref{estimator: fourier}) exactly cancels the characteristic function $\EE\big[e^{i\bzeta^T\bar{\bep}}\big]$ of the noise which makes $\hat{f}_R(\bar{\bx})$ an unbiased
estimator of the analytical part of $f$ and yet the remainder can be uniformly small. Thus, overall bias reduction is achieved.

In addition to the setting when $\btheta\in\RR^d$ resides in a finite-dimensional space, there are also some progress recently in the related problems in nonparametric domain where $\theta:\RR^d \rightarrow \RR$ belongs to some multivariate density class and $f(\theta)$ is a known functional of $\theta$. For instance, \cite{berrett2019efficient} studied the estimation of differential entropy of $\theta$:
$f(\theta)=-\int_{\bx\in\calX} \theta(\bx) \log\theta(\bx) d\bx$ when $\theta$ belongs to certain H\"{o}lder-type class. The authors generalized the idea of
the Kozachenko-Leonenko estimator \cite{zbMATH04030688} by introducing a deliberately constructed weighted version, and proved its efficiency and asymptotic normality for arbitrary $d$ while relaxed the support of $\theta$ to be unbounded. In another work, \cite{han2020optimal} established the minimax optimal rate of the same problem over certain general Lipschitz balls with smoothness index $s\in(0,2]$.
Their results hold when $\theta$ doesn't need to be bounded away from zero and can have unbounded support. Their approach was based on using kernel smoothing estimator $\theta_h$ of $\theta$, and then to estimate $f(\theta)$ by constructing an estimator of $f(\theta_h)$ with polynomial approximation, Taylor expansion and $U$-statistic techniques.

\subsection{Our contribution}
Although the methods we mentioned above are versatile and feasible in each of their settings, they are also either computationally intensive when it comes to implementation or limited by the form of the functional and/or the distribution type of the noise. In this article, inspired by \eqref{estimator: fourier}, we go back to the classical multivariate Taylor expansion and leverage the use of degenerate $U$-statistic, and propose a new method which we call High-Order Degenerate Statistical Expansion (HODSE). The new method leads to a unified estimator of $f(\btheta)$ with a remarkably neat explicit formula for high-order expansions.
What's more important, the resulting new estimator miraculously and systematically
cancels the bias and all other non-degenerate terms up to any preassigned order in
estimating the intermediate derivatives when using classical Taylor series to reduce bias,
with an explicit Riemann-Liouville integral formula for the remainder term
as in the analytical Taylor expansion.  This makes HODSE a natural
statistical version of the classical Taylor expansion.

The paper is organized as follows.
In Section 2, we introduce HODSE with the explicit formula for its degenerate expansion,
explain its differences as well as connections to the classical Taylor series which is commonly adopted,
and discuss its connection to jackknife and bootstrap.
In Section 3, we establish upper bounds on the $L_p$- risk ($p\geq 1$) for the estimation of $f(\btheta)$ under H\"{o}lder smoothness with only minimal moment constraints on the noise. These generalize the previous results in \cite{koltchinskii2021} under Gaussian shift model using iterative bootstrap which are shown to be asymptotically efficient and minimax optimal when $\scrH = \RR^d$ and $r(\bSigma) = d$. In Section 4, we establish the asymptotic normality of the HODSE estimator under only finite second moment of the noise, including the cases where higher moments appear in the statistical expansion.
In Section 5, we apply HODSE to the estimation of
$f(\btheta) = \|\btheta\|_p^p/d$ in the non-smooth case $0< p \le 1$ to generalize the minimax rate optimal estimation results in the Gaussian shift model \cite{cai2011} to a wider class of noise distributions. Moreover, it shows that the power of HODSE is not limited to smooth functional estimation, but also it can be applied to non-smooth functional estimation.

To summarize, we proposed a new method that naturally
provides a statistical version of the classical Taylor expansion.
The resulting estimator not only yields a neat and elegant formula, but also can be applied to a broad range of smooth and non-smooth functional estimation problems and thus achieves universality.

\subsection{Preliminaries and Notations}
\label{section: prelim}

Throughout this paper,
we assume for simplicity that the domain $\scrH$ of $f(\bx)$ is a Euclidean or separable Hilbert space. Let $\langle\cdot,\cdot\rangle$ be the inner product of $\scrH$ and $\|\cdot\|$ the induced norm, e.g. $\langle\bu,\bv\rangle = \bu^\top\bv$ and $\|\bu\|=\|\bu\|_2$ for $\scrH=\R^d$.
We use boldface lowercase letter $\bx$ to denote a member of $\scrH$ and boldface uppercase letter $\bSigma$ to denote a bounded linear operator in $\scrH$, e.g. the covariance matrix of a random vector for $\scrH=\R^d$.
We denote by $\otimes_{\ell=1}^k\bu_\ell = \bu_1\otimes\cdots\otimes \bu_k$ the Kronecker product of $k$ members of $\scrH$ as a rank-one $k$-linear form $(\bv_1,\ldots,\bv_k)\to \prod_{\ell=1}^k \langle \bu_\ell,\bv_\ell\rangle$, and $\scrH^{\otimes k}$ the set of all continuous $k$-linear forms in $\scrH$ as  the completion of all the linear combinations of the rank-one $k$-linear forms. The space $\scrH^{\otimes k}$ is equipped with the inner-product $\langle\cdot,\cdot\rangle_k$ determined by the bilinear extension of $\langle\otimes_{\ell=1}^k \bu_\ell,\otimes_{\ell=1}^k\bv_\ell\rangle_k = \prod_{\ell=1}^k \langle \bu_\ell,\bv_\ell\rangle$ in $\scrH^{\otimes k}$, the Hilbert-Schmidt norm $\|\cdot\|_{\rm HS}$ induced by the inner-product $\langle\cdot,\cdot\rangle_k$, and the tensor spectral norm
\bes
\big\|T^{(k)}\big\|_{\rm S} := \max_{\|\bh_1\| = \cdots = \|\bh_k\|=1} \Big|\big\langle T^{(k)}, \otimes_{j=1}^k \bh_j\big\rangle_{k}\Big|,\
T^{(k)}\in \scrH^{\otimes k}.
\ees
A member $T^{(k)}$ in $\scrH^{\otimes k}$ is symmetric if $\langle T^{(k)},\otimes_{\ell=1}^k \bv_\ell \rangle_k = \langle T^{(k)},\otimes_{\ell=1}^k\bv_{i_\ell}\rangle_k$ for all permutations $(i_1,\ldots,i_k)$ of $(1,\ldots,k)$. We denote by $\otimes$ the Kronecker product as the bilinear mapping from $(\scrH^{\otimes j},\scrH^{\otimes k})$ to $\scrH^{\otimes(j+k)}$ determined by $(\otimes_{\ell=1}^j\bu_\ell,\otimes_{\ell=1}^k\bv_\ell)\to \otimes_{\ell=1}^j\bu_\ell \otimes_{\ell=1}^k\bv_\ell$. We denote by $\times_j$ the mode $j$ tensor product determined by
$T^{(k)}\times_j \bSigma: \otimes_{\ell=1}^k \bv_\ell
\to \langle T^{(k)},\otimes_{\ell=1}^{j-1} \bv_\ell\otimes (\bSigma\bv_j)\otimes_{\ell=j+1}^k \bv_\ell\rangle_k$ as a $k$-linear form, and
$T^{(k)}\times_j \bu: \otimes_{\ell=1}^{j-1} \bv_\ell \otimes_{\ell=j+1}^k \bv_\ell
\to \langle T^{(k)},\otimes_{\ell=1}^{j-1} \bv_\ell\otimes \bu \otimes_{\ell=j+1}^k \bv_\ell\rangle_k$ as a $(k-1)$-linear form.

Unless otherwise stated we assume $f : \scrH\rightarrow \RR$ is m-times Fr\'echet differentiable at $\btheta$ with $f^{(k)}(\btheta)$ being the $k$-th derivative as a
symmetric $k$-linear form in $\scrH^{\otimes k}$ with finite spectral norm $\|f^{(k)}(\btheta)\|_S$,
$1\leq k \leq m$. For any $s>0$, the H\"older norm of order $s$ is defined with  $m = \lceil s\rceil-1$ by
\bel{Holder-s}
\|f\|_{(s)} := \sup_{\bx, \by\in \scrH;\|\bx-\by\|\neq 0}
\bigg\{\frac{\|f^{(m)}(\bx) - f^{(m)}(\by)\|_{\rm S}}{\|\bx-\by\|^{s-m}}\bigg\}.
\eel
We may also use the H\"{o}lder smoothness of $f$ induced by the Hilbert-Schmidt norm to
\bes
\|f\|_{(s),\rm HS} := \sup_{\bx,\by\in \scrH;\|\bx-\by\|\neq 0}\bigg\{\frac{\|f^{(m)}(\bx) - f^{(m)}(\by)\|_{\rm HS}}{\|\bx-\by\|^{s-m}}\bigg\}.
\ees
For notational simplicity, we suppress the dependence
of the spectral and Hilbert-Schmidt norms on the tensor order $k$.
We note that when $\scrH=\RR$, $\|f^{(k)}(x)\|_{\rm S} = \|f^{(k)}(x)\|_{\rm HS} = |f^{(k)}(x)|$
and $\|f\|_{(s),HS}=\|f\|_{(s)}$.

For $p>0$, we denote the $\ell_p$ norm by
$\|v\|_p = (\sum_{j=1}^d |v_j|^p)^{1/p}$ for $v =(v_1,\ldots,v_d)^\top \in\R^d$ and
the $L_p$ norm by $\|X\|_{L_p} = \big(\EE\big[|X|^p\big]\big)^{1/p}$
for random variables $X\in\RR$ and
$\|\bep\|_{L_p} = \big(\EE\big[\|\bep\|^p\big]\big)^{1/p}$ for random elements $\bep\in\scrH$.

We use $\mathcal{F}$ and $\mathcal{F}^{-1}$ to denote the Fourier transform (FT) and inverse Fourier transform (IFT) respectively.
We use the conventional notation ``$\Rightarrow$"  to denote weak convergence, i.e. convergence in distribution and use ``$\xrightarrow{p}$" to denote convergence in probability.
Throughout the paper, given nonnegative numbers $a$ and $b$, $a\lesssim b$ means that $a\leq Cb$ for a numerical constant $C$, and $a \asymp b$ means that $a \lesssim b$ and $b \lesssim a$,
$a \wedge b = \min \{ a, b\}$ and $a \vee b = \max \{ a, b\}$. We use $C_{a}$ to denote a constant depending on $a$ only.

\section{High-order degenerate statistical expansion}\label{sec:method}

We are interested in the estimation of $f(\btheta)$ with a known function $f$ and
a high- or infinite-dimensional unknown parameter $\btheta$ based on observations $\bx_j$ with mean $\EE[\bx_j]=\btheta$.
Throughout this section we assume that
$f(\cdot): \scrH\to\RR$ is $m$ times
Fr\'echet differentiable at $\btheta$ so that the following expansion holds,
\bes
f(\btheta + \bh) = f(\btheta) + \sum_{k=1}^m \frac{\big\langle f^{(k)}(\btheta), \bh^{\otimes k}\big\rangle_k}{k!}+ o(\|\bh\|^m)
\ees
for symmetric $k$-linear forms $f^{(k)}(\btheta)$ with finite spectral norm $\big\|f^{(k)}(\btheta)\big\|_{\rm S}$, $1\le k\le m$.

Consider the estimation of $f(\btheta)$ based on independent data points $\bx_j\in\scrH, 1\le j\le n$, with $\E[\bx_j]=\btheta$.
Let $\bep_j=\bx_j-\btheta$, $\bar{\bx} = \sum_{j=1}^n \bx_j/n$ and $\bar{\bep} = \bar{\bx}-\btheta$.
The plug-in estimator $f(\bar{\bx})$ has the Taylor expansion
\bel{new-exp-1}
f(\bar{\bx}) = f(\btheta) + \sum_{k=1}^m \frac{\big\langle f^{(k)}(\btheta), \bar{\bep}^{\otimes k}\big\rangle_k}{k!}
+ O(1)\|f\|_{(s)} \|\bar{\bep}\|^s
\eel
with the H\"older norm $\|f\|_{(s)}$ in \eqref{Holder-s}.
When the remainder term is sufficiently small for $1<s\le 2$,
$f(\bar{\bx})$ is asymptotically linear and the CLT holds under the Lindeberg condition.
We note that the Lindeberg condition is needed in our setting
even when $\bep_j$ are i.i.d. with zero mean and finite variance
because the distribution of $\big\langle f^{(1)}(\btheta), \bep_j\big\rangle_1$ typically changes
when $\btheta \in \RR^d$ and $d = d_n\to\infty$.

Let $\bSigma = \EE[n^{-1}\sum_{j=1}^n \bep_j\otimes\bep_j]$,
$\sigma = \|\bSigma\|_{\rm S}^{1/2}$ and $r(\bSigma) = \hbox{\rm trace}(\bSigma)/\sigma^2\ge 1$ be the effective rank of $\bSigma$,
where $\trace(\bSigma) = \langle \bI_{\scrH}, \bSigma\rangle_2$ with the identity operator $\bI_{\scrH}$ in $\scrH$.
The plug-in estimator can be biased
when we use higher order approximations, say $s>2$ in \eqref{new-exp-1}, to control the remainder term.
For example, by comparing the usual bounds for the standard deviation of the linear term and
the bias in the second order expansion when $\bep_j$'s are i.i.d. copies of an isotropic $\bep$,
\bes
\big\{\EE\big[\big\langle f^{(1)}(\btheta), \bar{\bep}\big\rangle_1^2\big]\big\}^{1/2}
&=& \big\langle f^{(1)}(\btheta)\otimes f^{(1)}(\btheta),\bSigma/n\big\rangle_2^{1/2}
\asymp \|f^{(1)}(\btheta)\|\sigma/n^{1/2},\
\cr \EE\big[\big\langle f^{(2)}(\btheta), \bar{\bep}^{\otimes 2}\big\rangle_2 \big]
&=& \big\langle f^{(2)}(\btheta), \bSigma\big\rangle_2/n
\le \|f^{(2)}(\btheta)\|_{\rm S}\sigma^2r(\bSigma)/n,
\ees
we find that the bias is typically of smaller order than the linear term when
\bes
\|f^{(2)}(\btheta)\|_{\rm S}\big(\sigma r(\bSigma)/n^{1/2}\big) \ll \|f^{(1)}(\btheta)\|,
\ees
which requires $r(\bSigma)\ll n^{1/2}$ when $\|f^{(2)}(\btheta)\|_{\rm S}$ and $\|f^{(1)}(\btheta)\|$
are of the same order.
Thus, bias correction may be needed when $r(\bSigma)\gg n^{1/2}$.

To remove the bias in the case of $s>2$, one may
subtract from $f(\bar{\bx})$ some estimator $\widehat{\bias_{\,2}}$ for
$\bias_{\,2} = \big\langle f^{(2)}(\btheta), \sum_{j=1}^n \bep_j^{\otimes 2}/(2n^2)\big\rangle_2$ in
the expansion of $\big\langle f^{(2)}(\btheta), \bar{\bep}^{\otimes 2}\big\rangle_2/2$.
For $s\ge 3$, we also want to remove
$\big\langle f^{(3)}(\btheta), \sum_{j=1}^n \bep_j^{\otimes 3}/(6n^3)\big\rangle_3$
and the third order bias term in the expansion of $-\widehat{\bias_{\,2}}$, so on and so forth.
This approach was discussed in detail in \cite{jiao2020} in the one dimensional case
in the Bernoulli model. As the authors of \cite{jiao2020} noted, even in the one dimensional case,
such an ad hoc process can be tedious and hard to keep track of when dealing with higher-order correction terms. Yet our case is more delicate and can be high dimensional.
One may also work with the Taylor expansion
\bes
f(\btheta) = f(\bar{\bx}) + \sum_{k=1}^m
\frac{\big\langle f^{(k)}(\bar{\bx}), \bar{\bep}^{\otimes k}\big\rangle_k}{(-1)^k k!}
+ O(1)\|f\|_{(s)} \|\bar{\bep}\|^s
\ees
at the observable $\bar{\bx}$ but similar issues arise when we deal with the unobservable $\bar{\bep}^{\otimes k}$.

In this paper, we present a unified yet very neat formula for bias correction which was aimed at achieving
\bel{new-est-1}
\widehat{f} \approx f(\btheta) + \sum_{k=1}^m \frac{\big\langle f^{(k)}(\btheta), \bar{\bep}^{(k)}\big\rangle_k}{k!},
\eel
with a remainder term of higher order than $m$, where
\bel{new-eps-(k)}
\bar{\bep}^{(k)} := \sum_{1\le j_1\neq \cdots \neq j_k\le n}
\frac{\bep_{j_1}\otimes \cdots\otimes \bep_{j_k}}{n(n-1)\cdots(n-k+1)}
\eel
are completely degenerate $U$-tensors of order $k$.
As $\bar{\bep}^{(k)}$ are completely degenerate,
the bias of \eqref{new-est-1}, $\EE\big[\widehat{f} - f(\btheta)\big]$,
is identical to the expectation of the remainder term in this expansion
and thus is expected to be of higher order than the $m$-th order.
An estimator of form \eqref{new-est-1}
may also achieve variance reduction by removing terms like
$\big\langle f^{(3)}(\btheta), \sum_{j_1\neq j_2} \bep_{j_1}^{\otimes 2}\otimes \bep_{j_2}/(6n^3)\big\rangle_3$
in the expansion of the plug-in estimator $f(\bar{\bx})$.

Given the order $s$ of the H\"older smoothness of $f$,
we propose to estimate $f(\btheta)$ by
\bel{new-est-2}
\widehat{f} : = f(\bar{\bx}) + \sum_{1\le k\le m} \frac{\langle f^{(k)}(\bar{\bx}),\bar{\bu}^{(k)}\rangle_k}{k!}
\eel
with $m = \lceil s\rceil-1$ and $U$-statistics
\bel{new-u-(k)}
\bar{\bu}^{(k)} := \sum_{1\le j_1\neq \cdots \neq j_k\le n}
\frac{(\bx_{j_1}-\bar{\bx})\otimes \cdots\otimes(\bx_{j_k} - \bar{\bx})}{n(n-1)\cdots(n-k+1)}.
\eel
We note that $\bar{\bu}^{(1)}={\bf 0}$ so that the sum in \eqref{new-est-2} actually runs from $k=2$ to $k=m$.

We advocate the use of \eqref{new-est-2},
a plug-in estimator of the right-hand side of \eqref{new-est-1},
as the proper statistical expansion of the plug-in estimator $f(\bar{\bx})$ to higher orders.
A remarkable feature of this natural formula is that the resulting estimator
automatically cancels out the impacts of estimating $f^{(k)}(\btheta)$ by $f^{(k)}(\bar{\bx})$
and $\bep_j$ by $\bx_j-\bar{\bx}$ on the right-hand side of \eqref{new-est-1}
simultaneously for all orders $k = 1,\ldots, m$.
Thus, the estimator \eqref{new-est-2} can be viewed as a degenerate statistical expansion
of function $f$.
This claim is formally justified by the following proposition which gives an explicit formula
for the remainder term in \eqref{new-est-1} when the estimator \eqref{new-est-2} is used.

\begin{restatable}{proposition}{PropStatTaylor}
\label{prop-1}
Let $\widehat{f}$ be as in \eqref{new-est-2} with $n\ge m \ge 2$.
Then, \eqref{new-est-1} holds,
\bel{new-prop-1-1}
\widehat{f} = f(\btheta) + \sum_{k=1}^m \frac{\big\langle f^{(k)}(\btheta), \bar{\bep}^{(k)}\big\rangle_k}{k!} - \Rem_m,
\eel
with the $\bar{\bep}^{(k)}$ in \eqref{new-eps-(k)} for the noise vectors $\bep_j=\bx_j-\btheta$
and with the remainder term
\bel{new-prop-1-2}
&& \Rem_m := \sum_{k=0}^{m}
\frac{\langle J^{m-k}\Delta^{(m)},
\bar{\bep}^{\otimes (m-k)}\otimes \bar{\bep}^{(k)}\rangle_m}{(-1)^{m-k}k!},
\eel
where $J^0h = h(1)$ and $J^\alpha h = \int_0^1 h(t)(1-t)^{\alpha -1}dt/\Gamma(\alpha)$
give the Riemann-Liouville integral operator
for $h: [0,\infty)\to \scrH^{\otimes m}$ and $\alpha>0$,
$\Delta^{(m)}(t) :=f^{(m)}(\bar{\bx} + t(\btheta-\bar{\bx})) - f^{(m)}(\bar{\bx})$ and
$\bar{\bep} = n^{-1}\sum_{j=1}^n\bep_j = \bar{\bx}-\btheta$.
\end{restatable}

While the expansion \eqref{new-prop-1-1} looks similar to the expansion \eqref{new-exp-1},
the difference is highly significant in high-order analysis
because $\bar{\bep}^{(k)}$ in \eqref{new-eps-(k)} are completely degenerate $U$-tensors.
For i.i.d. data, $\EE[\bar{\bep}^{(k)}]={\bf 0}$ for all integers $k\ge 1$ but
$\EE[\langle \bv^{\otimes 2k}, \bar{\bep}^{\otimes 2k}\rangle_{2k}] > 0$ for all $\bv\neq 0$.
Moreover, $\EE\big[\big\langle f^{(k)}(\btheta), \bar{\bep}^{(k)}\big\rangle_k^2\big]$ depends
on $\bep$ only through its covariance operator for all $k$ [See \eqref{var-S_k}].

We call \eqref{new-est-2} High-Order Degenerate Statistical Expansion (HODSE)
as it gives a natural statistical version of the Taylor expansion at
the true parameter $\btheta$ in the following sense: The $k$-th order term on the right-hand side of \eqref{new-prop-1-1} has
zero mean and a standard deviation of the same form as the absolute value of the $k$-th
order term in the analytical Taylor expansion.
For $f: \R\to\R$ and independent noise with common standard deviation $n^{1/2}\sigma_n$,
the standard deviation of the $k$-th order term in \eqref{new-prop-1-1} is
$(1+o(1))|f^{(k)}(\theta)| \sigma_n^k/k!$
whereas the absolute value of the $k$-th order term in the Taylor series of $f(\theta+t)$
is $|f^{(k)}(\theta)|\times |t|^k/k!$.

In the following lemma we provide some analytic upper bounds for the remainder term.
For symmetric order $m$ tensors $T^{(m)}$, define norms
\bel{mixed-norm}
\|T^{(m)}\|_{m-k,k} := \sup_{\|\bu\|=\langle \bv,\bv\rangle_k =1}
\langle T^{(m)},\bu^{\otimes (m-k)}\otimes \bv\rangle_m,
\eel
where the supreme is taken over $\bu\in \scrH$ and order $k$ tensors $\bv$.
In particular, $\|T^{(m)}\|_{m,0}$ is no greater than the spectral norm $\|T^{(m)}\|_{\rm S}$
and $\|T^{(m)}\|_{0,m}$ is the Hilbert-Schmidt norm
$\|T^{(m)}\|_{\rm HS} = \langle T^{(m)},T^{(m)}\rangle_m^{1/2}$.

\begin{restatable}{lemma}{LemmaAnalytical}
\label{lm-1}
Let $n\ge m = \lceil s\rceil-1\ge 2$ and $\Rem_m$ be the remainder term in \eqref{new-prop-1-2}. Then,
\bel{lm-1-1}
|\Rem_m|
\le \sum_{k=0}^{m}
\bigg(\max_{0 < t\le 1}\frac{\|\Delta^{(m)}(t)\|_{m-k,k}}{(t\|\bar{\bep}\|)^{s-m}}\bigg)
\frac{
\|\bar{\bep}\|^{s-k}\langle \bar{\bep}^{(k)},\bar{\bep}^{(k)}\rangle_k^{1/2}}{\Gamma(s-k+1) k!}
\eel
with the tensor norm $\|\cdot\|_{m-k,k}$ in \eqref{mixed-norm} and
\bel{lm-1-2}
\qquad |\Rem_m|
&\le& \bigg(\max_{0< t\le 1}\frac{\|\Delta^{(m)}(t)\|_{\rm S}}{(t\|\bar{\bep}\|)^{s-m}}\bigg)
\sum_{k_1+\ldots + k_b = \ell \le k\le m \atop k_a\ge 2\,\forall a; b\ge 0}
\frac{
(C_{k,n}/C_{b,n})C^{(k)}_{k_1,\ldots,k_b}}{\Gamma(s-k+1) k!}
\\ \nonumber && \qquad\qquad \times \sum_{j_1\neq \ldots \neq j_b}
\frac{\|\bep_{j_1}\|^{k_1}\cdots\|\bep_{j_b}\|^{k_b}\|\bar{\bep}\|^{s-\ell}}
{n(n-1)\cdots(n-b+1)n^{\ell-b}}
\eel
with $C_{k,n} = n^k(n-k)!/n!$ satisfying $C_{k,n} \le \exp((k-1)k/n)$
and certain positive integers $C^{(k)}_{k_1,\ldots,k_b}$ satisfying
$r_k =\sum_{k_1+\ldots + k_b = \ell \le k \atop k_a\ge 2\,\forall a; b\ge 0}
C^{(k)}_{k_1,\ldots,k_b}  \le k!$.
\end{restatable}

In the above lemma, \eqref{lm-1-1} is typically sharper than \eqref{lm-1-2} when $\scrH$ is of low-dimension,
and vice versa.
For example, when $\scrH = \RR^d$, $\Delta^{(m)}(t)$ is a $d\times\cdots\times d$ tensor,
so that $\|\Delta^{(m)}(t)\|_{m-k,k} \le d^{k-1}\|\Delta^{(m)}(t)\|_{\rm S}$
and \eqref{lm-1-1} could be sharper when $d^{m-1}\le m!$.
In \eqref{lm-1-2}, the $U$-tensor $\bar{\bep}^{(k)}$ needs to be written
as a sum of rank-one tensors of proper forms to apply the spectrum norm,
and $r_k$ is the number of different types of such rank-one tensors.
The order of constants $1/\{(m-k)!k!\}$ and $r_k$ in \eqref{lm-1-1} and \eqref{lm-1-2},
arising from the explicit remainder formula,
become crucial
when we need to use the estimator \eqref{new-est-2} with diverging order $m$,
see for example, Subsections 5.1 and 5.2. In our study of the asymptotic normality in Theorem \ref{th-1},
both \eqref{lm-1-1} and \eqref{lm-1-2} are used respectively under conditions on
the ``effective rank'' of $f$ and the boundedness of the order $m$ of the expansion.

\subsection{Jackknife and bootstrap}
{
The jackknife and bootstrap approaches converge to the Taylor expansion approach in HODSE.
The following discussion casts the estimator \eqref{new-est-2} as a plug-in jackknife estimator of
\bel{new-est-3}
f(\btheta) + \sum_{k=1}^m \frac{\big\langle f^{(k)}(\btheta), \bar{\bep}^{(k)}\big\rangle_k}{k!},
\eel
ideally with smaller estimation error for larger $m$.
Let $\EE^*$ be the conditional expectation given the data points $\{\bx_1,\ldots,\bx_n\}$ in $\scrH$, and  $\{\bx^*_1,\ldots,\bx^*_m\}$ be sampled without replacement from $\{\bx_1,\ldots,\bx_n\}$ under $\EE^*$.  Let $\btheta^*=\bar{\bx}$ be the mean of the resampled data and
$\bep^*_j=\bx^*_j- \btheta^*$ the resampled noise.
Because
\bes
\EE^*\big[\bep^*_1  \otimes \cdots \otimes  \bep^*_k\big] = \bar{\bu}^{(k)}
\ees
with the $\bar{\bu}^{(k)}$ in \eqref{new-u-(k)}, the estimator \eqref{new-est-2} can be written as
\bel{jackknife}
\widehat{f} &=& f(\bar{\bx}) + \sum_{k=1}^m \frac{\big\langle f^{(k)}(\bar{\bx}),
\bar{\bu}^{(k)}\big\rangle_k}{k!}.
\\ \nonumber &=& \EE^*\bigg[f(\btheta^*) + \sum_{k=1}^m \frac{\big\langle f^{(k)}(\btheta^*),
\bep^*_1\otimes \cdots \otimes \bep^*_k\big\rangle_k}{k!}\bigg].
\eel
This can be viewed a jackknife estimator of \eqref{new-est-3}
as subsamples of size $k$ are used to correct the biases of the estimates
for $\big\langle f^{(\ell)}(\btheta), \bar{\bep}^{(\ell)}\big\rangle_\ell$ based on subsamples of sizes
$\ell = 1,\ldots, k-1$, and the average $\E^*$ over all permutations is used to achieve noise cancellation.
It is different from the standard Jackknife as the same $\btheta^*=\bar{\bx}$ is plugged-in
in the estimation of $f^{(k)}(\btheta)$ to achieve systematic bias cancellation.
Bootstrap can be used to compute the estimator \eqref{new-est-2},
\bel{jackknife-bootstrap}
\widehat{f} \approx f(\bar{\bx}) + \sum_{k=1}^m \sum_{B=1}^{N_K} \frac{\big\langle f^{(k)}(\bar{\bx}),
\bep^*_1\otimes \cdots \otimes \bep^*_k\big\rangle_k^{(B)}}{k! N_K},
\eel
where $\big\langle f^{(k)}(\bar{\bx}),
\bep^*_1\otimes \cdots \otimes \bep^*_k\big\rangle_k^{(B)}, 1\le k\le m$,
are computed using $\{\bep_1^*,\ldots,\bep^*_m\}^{(B)}$,
and $\{\bep_1^*,\ldots,\bep^*_m\}^{(B)}$ are  i.i.d. copies of
 $\{\bep_1^*,\ldots,\bep^*_m\}$ under $\EE^*$.
We note that this bootstrap of the jackknife estimator \eqref{jackknife}
is different from the standard empirical bootstrap \cite{efron1982jackknife} as $\bep^*_1, \ldots, \bep^*_m$
are sampled without replacement for each bootstrap copy $\{\bep_1^*,\ldots,\bep^*_m\}^{(B)}$
but the bootstrap copies $\{\bep_1^*,\ldots,\bep^*_m\}^{(B)}$ for different $B$ are sampled with replacement.

It is remarkable that the plug-in jackknife estimator \eqref{jackknife} and its bootstrapped version
\eqref{jackknife-bootstrap}, with the population mean $\btheta$ replaced by the jackknife/bootstrap mean
$\bar{\bx}$ simultaneously in the nonlinear $f^{(k)}(\cdot)$,
automatically yields the non-degenerate expansion \eqref{new-est-3} with a remainder
term of higher order than $m$, akin to the automatic Edgeworth expansion adjustments
of resampling methods in the central limit theorem
\cite{singh1981asymptotic, bickel1981some, wu1986jackknife}.
}

\section{Risk bounds}\label{sec:risk}
We provide in this section upper bounds for the bias and $L_1(\P)$ and $L_2(\P)$ risks of the estimator \eqref{new-est-2}.
Unless otherwise stated, we assume in the rest of the paper that $\bep_j$
are independent but not necessarily identically distributed random elements in $\scrH$.
Recall that the covariance operator $\bSigma$
for the noise, noise level $\sigma$ and the effective rank of $\bSigma$ are defined respectively as
\bel{effective-rank}
\bSigma = \EE\bigg[\frac{1}{n}\sum_{j=1}^n \bep_j\otimes\bep_j\bigg],\quad \sigma = \big\|\bSigma\big\|_{\rm S}^{1/2},\quad
r = r(\bSigma) = \frac{\trace(\bSigma)}{\sigma^2}.
\eel

Because $\EE\big[\big\langle f^{(k)}(\btheta), \bar{\bep}^{(k)}\big\rangle_k\big]=0$ in the expansion
\eqref{new-prop-1-1}, the bias of the estimator \eqref{new-est-2} is no greater in absolute value than
the $L_1(\P)$ norm of the remainder term in \eqref{new-prop-1-2}. In the following theorem, we develop
such bias bounds under a H\"older smoothness condition on $f$.

\begin{restatable}{theorem}{TheoremBias}
\label{TheoBias}
Let $\widehat{f}$ be the estimator \eqref{new-est-2} with $n\ge m = \lceil s\rceil-1\ge 2$
and $\Rem_m$ be the remainder term in \eqref{new-prop-1-2}.
Let $p\ge 1$ and $C^*_{m,n} = \sum_{k=0}^{m} n^k(n-k)!/\{(m-k)! n!\}$. Then,
\bel{TheoBias-1}
\|\Rem_m\|_{L_p(\P)}
\le C_{m,n}^*\|f\|_{(s)}
\max\bigg\{\|\bar{\bep}\|_{L_{ps}(\P)}^{s}, \bigg\|\sum_{j=1}^n \frac{\|\bep_j\|^2}{n^2}\bigg\|^{s/2}_{L_{ps/2}(\P)}\bigg\}
\eel
for any (possibly with nonzero mean and dependent) random $\bep_j = \bx_j - \btheta$.
Moreover, if $\bep_j$ are independent with $\EE[\bep_j]={\bf 0}$,
then for some constant $C^*_{ps}$ depending on $p,s$ only
\bel{TheoBias-2}
\qquad && \big|\EE\big[\widehat{f}\big] - f(\btheta)\big| \le \|\Rem_m\|_{L_p(\P)}
\le C^*_{ps}\|f\|_{(s)} \bigg\{\big\|\bar{\bep}\|_{L_2(\P)}^s
+ \bigg(\sum_{j=1}^n \frac{\|\bep_j\|_{L_{ps}(\P)}^{ps}}{n^{ps}}\bigg)^{1/p}\bigg\}.
\eel
\end{restatable}

We note that the $L_1(\P)$ norm of the remainder in \eqref{new-prop-1-2}
is sufficient to bound the bias.
The $L_p(\P)$ bounds in Theorem~\ref{TheoBias} will be useful
in the development below of bounds for the $L_p(\P)$ risk of $\widehat{f}$.
It follows from the upper bound for $C_{k,n}$ in Lemma \ref{lm-1} that
$C^*_{m,n} \le e^{1+(m-1)m/n}$.

Theorem \ref{TheoBias} requires $s$-th moment of the noise for bias correction of the same order.
For i.i.d. $\bep_j$ and $p=1$, $\big\|\bar{\bep}\|_{L_2(\P)}^2 = \|\bep_1\|_{L_2(\P)}^2/n$
so that \eqref{TheoBias-2} can be written as
\bel{TheoBias-2-iid}
\quad && \big|\EE\big[\widehat{f}\big] - f(\btheta)\big| \le \|\Rem_m\|_{L_1(\P)}
\le C^*_s\|f\|_{(s)}
\bigg\{\frac{\big\|\bep_1\big\|_{L_{2}(\P)}^s}{n^{s/2}} + \frac{\|\bep_1\|_{L_s(\P)}^s}{n^{s-1}}\bigg\},
\eel
and the $s$-th moment term is subsumed into the second moment term $\big\|\bar{\bep}\big\|_{L_{2}(\P)}^s$ iff
\bel{s-th-moment}
\left\|\bep_1/\|\bep_1\|_{L_2(\P)}\right\|_{L_s(\P)} = O(n^{1/2-1/s}).
\eel

In the upper bounds in Theorem \ref{TheoBias}, we may replace $\|f\|_{(s)}$ by $2\|f\|_{(s),\btheta}$ where
\bel{local-Holder}
\|f\|_{(s),\btheta} = \sup_{\|\bh\|>0}\big\|f^{(m)}(\btheta + \bh)-f^{(m)}(\btheta)\big\|_{\rm S}/\|\bh\|^{s-m}
\eel
with $m = \lceil s\rceil-1$. As $C^*_s$ is implicit,
Theorem \ref{TheoBias} is most useful when $\scrH$ is high- or infinite-dimensional but $s >0$ is fixed.

The $L_p(\P)$ bounds for the remainder in \eqref{TheoBias-2} or \eqref{TheoBias-2-iid}
naturally lead to $L_p(\P)$ risk bounds for $\widehat{f}$
given any $L_p(\P)$ bound for the completely degenerate $U$ variables
$\big\langle f^{(k)}(\btheta), \bar{\bep}^{(k)}\big\rangle_k$ in \eqref{new-prop-1-1}.
However, as we are interested in the case where bias correction with $s>2$ is needed to improve
the naive plug-in estimator $f(\bar{\bx})$, the following $L_2(\P)$ bounds will be used.

Write the degenerate $U$-variables in the expansion \eqref{new-prop-1-1} as a sum of uncorrelated terms,
\bel{S_k}
S_k = \big\langle f^{(k)}(\btheta), \bar{\bep}^{(k)}\big\rangle_k
= \sum_{1\le j_1 < \cdots < j_k\le n}
\frac{\langle f^{(k)}(\btheta),\bep_{j_1}\otimes \cdots \otimes \bep_{j_k}\rangle_k}{n(n-1)\cdots(n-k+1)/k!}.
\eel
For $k\ge 2$ and i.i.d. $\bep_j$, the variance of $S_k$ can be explicitly expressed as
\bel{var-S_k}
\qquad && \Var(S_k) = \EE\big[S_k^2\big]
= (C_{k,n}k!)V_k/n^k,
\eel
with the constant $C_{k,n} = n^k(n-k)!/n!\le e^{(k-1)k/n}$ in Lemma \ref{lm-1} and
\bel{V_k}
V_k = \langle f^{(k)}(\btheta),  f^{(k)}(\btheta)\times_1\bSigma \times_2\cdots\times_k \bSigma \rangle_{k},
\eel
where $\bSigma$ is the covariance operator in \eqref{effective-rank}
and $\times_\ell$ denotes the mode-$\ell$ tensor product.
For independent not identically distributed (i.n.i.d.) noise $\bep_j$,
\bel{var-S_k-inid}
\qquad && \Var(S_k) = \EE\big[S_k^2\big]
\le (C_{k,n}^2k!)V_k/n^k.
\eel
With the tensor spectrum norm $\|\cdot\|_{\rm S}$ and the effective rank in \eqref{effective-rank},
$V_k$ is bounded by
\bel{var-S_k-bd}
V_k \le \|f^{(k)}(\btheta)\|_{\rm S}^2\sigma^{2k}(r(\bSigma))^{k-1}.
\eel

The following theorem is based on the $L_p(\P)$ bound \eqref{TheoBias-2} for the remainder
and the $L_2(\P)$ bounds for the degenerate $U$-tensors via \eqref{var-S_k-bd}.

\begin{restatable}{theorem}{TheoremRiskBd}
\label{th-risk-bd}
Let $n\ge m = \lceil s\rceil-1\ge 2\ge p\ge 1$ and
$\widehat{f}$ be as in \eqref{new-est-2} with
independent observations $\bx_j = \bep_j+\btheta$ with $\EE[\bep_j]={\bf 0}$.  Then,
\bel{th-risk-bd-1}
\quad \big\|\widehat{f} - f(\btheta)\big\|_{L_p(\P)}
&\le& \bigg(\sum_{k=1}^m \frac{C_{k,n}^2V_k}{n^kk!}\bigg)^{1/2}
+ C^*_{s}\|f\|_{(s)}\bigg(\frac{\sigma^2 r}{n}\bigg)^{s/2}
\\ \nonumber && + C^*_{s}\|f\|_{(s)}\bigg(\sum_{j=1}^n \frac{\|\bep_j\|_{L_{ps}(\P)}^{ps}}{n^{ps}}\bigg)^{1/p}
\eel
with the $\sigma$ and $r=r(\bSigma)$ in \eqref{effective-rank}, $V_k$ in \eqref{V_k},
$C_{k,n} \le e^{(k-1)k/n}$ and a constant $C^*_{s}$ depending on $s$ only.
Moreover, if $\max_{2\le k <s} \|f^{(k)}(\btheta)\|_{\rm S}\le C_0$ and $\|f\|_{(s),\btheta}\le C_0$, then
\bel{th-risk-bd-2}
\quad \big\|\widehat{f} - f(\btheta)\big\|_{L_p(\P)}
&\le& (V_1/n)^{1/2} + C_0C^*_s\max\big\{\sigma^2\sqrt{r}/n, (\sigma^2 r/n)^{s/2}\big\}
\\ \nonumber &&  + C_0C^*_{s}\bigg(\sum_{j=1}^n \frac{\|\bep_j\|_{L_{ps}(\P)}^{ps}}{n^{ps}}\bigg)^{1/p}.
\eel
\end{restatable}

Theorem \ref{th-risk-bd} is proved in Section \ref{section: proofs} along with proofs of
\eqref{var-S_k}, \eqref{var-S_k-inid} and \eqref{var-S_k-bd}.
Compared with \cite{koltchinskii2021} in the literature where $\bep_j$ are assumed to be i.i.d. Gaussian,
Theorem \ref{th-risk-bd} requires the $ps$-th moment condition in the third component
on the right-hand side of \eqref{th-risk-bd-1}.
In the i.i.d. Gaussian case, we may write $\bep_1 = \sum_{\ell} \lam_\ell^{1/2} z_{\ell}\bv_\ell$
with i.i.d. $z_{\ell}\sim \calN(0,1)$
via the eigenvalue decomposition $\bSigma = \sum_\ell\lam_\ell \bv_\ell\otimes \bv_\ell$,
so that for even integers $ps\ge 2$,
\bel{term3}
\quad \sum_{j=1}^n \frac{\|\bep_j\|_{L_{ps}(\P)}^{ps}}{n^{ps}}
= n^{1-ps}\EE\bigg[\bigg(\sum_{\ell}\lam_\ell z_\ell^2\bigg)^{ps/2}\bigg]
\le n^{1-ps/2}(\sigma^2 r/n)^{ps/2}(ps-1)!!.
\eel
Especially, when $r=o(n)$ and we take $p=2$ and $s>1$, \eqref{term3} is of an asymptotic order smaller than $O((r/n)^{s/2}$). As a result, bound \eqref{th-risk-bd-2} reproduces the bound on MSE: $\|\widehat{f} - f(\btheta)\|^2_{L_2(\P)} \lesssim (n^{-1}\vee (r/n)^s)$, which was proved in \cite{koltchinskii2021} under the i.i.d. Gaussian assumption and shown to be minimax optimal under the standard Gaussian shift model.

\section{Asymptotic normality}
In this section we develop asymptotic normality theory for the proposed estimator \eqref{new-est-2}
of $f(\btheta)$.

We have developed upper bounds for the bias and $L_p$ risk when the
$ps$-moment of the noise does not grow too fast. The moment condition is natural for the $L_p$
risk with expansions of order $s$, and the risk bound can be used to remove higher order terms
in the asymptotic normality analysis. However, such an approach would
require higher moment condition than necessary.
In Theorem \ref{th-1} below, asymptotic normality requires only the Lindeberg
condition on the linear term and a mild condition on the growth rate of the second moment,
provided the H\"older smoothness of $f(\cdot)$ at $\btheta$.
In addition to the H\"older smoothness condition in the spectrum norm,
we shall consider its counterpart in the Hilbert-Schmidt norm,
\bes
\|f\|_{(s),{\rm HS},\btheta} = \sup_{\|\bh\|>0}\big\|f^{(m)}(\btheta + \bh)-f^{(m)}(\btheta)\big\|_{\rm HS}/\|\bh\|^{s-m}.
\ees

\begin{restatable}{theorem}{TheoremFixedOrder}
\label{th-1}
Let $n\ge m=\lceil s\rceil-1\ge 2$ and $\fhat$ be as in \eqref{new-est-2}
based on independent observations $\bx_j = \bep_j+\btheta \in \scrH$, $j\le n$,
with $\EE[\bep_j]={\bf 0}$.
Let $V_1 = \langle f^{(1)}(\btheta)\otimes f^{(1)}(\btheta),\bSigma\rangle_2$ be as in \eqref{V_k},
$\sigma$ and $r(\bSigma)$ be the noise level and effective rank as in \eqref{effective-rank},
and $\|f\|_{(s),\btheta}$ be as in \eqref{local-Holder}.
Suppose $\max_{2\le k <s} \|f^{(k)}(\btheta)\|_{\rm S}\le C_0$ and $\|f\|_{(s),\btheta}\le C_0$,
\bel{th-1-2}
C_0\max\big\{\sigma^2r^{1/2}(\bSigma)/n, (\sigma^2r(\bSigma)/n)^{s/2}\big\} \ll (V_1/n)^{1/2},\quad s = O(1),
\eel
and the Lindeberg condition holds for $\{\langle f^{(1)}(\btheta),\bep_j\rangle_1, j\le n\}$. Then,
\bel{th-1-3}
n^{1/2}\big(\widehat{f} - f(\btheta)\big)\big/V_1^{1/2}
\Rightarrow \calN\big(0,1\big)\quad \hbox{ as $n\to\infty$.}
\eel
The asymptotic normality \eqref{th-1-3} still holds when the condition $s=O(1)$ is weakened to
$s^2\le n$ in \eqref{th-1-2} and the smoothness condition on $f$ is
replaced by $\max_{2\le k <s} \|f^{(k)}(\btheta)\|_{\rm S}/d^{k/2} \le C_0$ in the spectrum norm and
$\|f\|_{(s),{\rm HS},\btheta}\le C_0d^{m/2}$ in the Hilbert-Schmidt norm with any fixed positive real number $d$.
\end{restatable}

In Theorem \ref{th-1}, all quantities, including $\scrH$, $f: \scrH\to \RR$, $\sigma$, $r(\bSigma)$ and $C_0$
are allowed to change with $n$.
While $d$ is just a constant in Theorem \ref{th-1}, for $\scrH=\R^d$
we have $\|f^{(k)}(\bx)\|_{\rm HS}^2 \le \|f^{(k)}(\bx)\|_{\rm S}^2d^{k - 1}$
and $\|f\|_{(s),{\rm HS},\btheta}^2 \le d^{m-1}\|f\|_{(s),\btheta}^2$ because
$f^{(m)}(\bx)\times_{k=1}^{m-2}\bfe_{i_k}$ is a $d\times d$ symmetric matrix given $\bx$ and canonical basis
vectors $\bfe_{i_1},\ldots,\bfe_{i_{m-2}}$, $1\le i_k\le d$.
For general $\scrH$, $d$ can be viewed as the effective rank of $f(\cdot)$.
Thus, the last statement of Theorem \ref{th-1} asserts that
the boundedness condition $s=O(1)$ in \eqref{th-1-2} on the order of the expansion
can be replaced by the boundedness condition $d=O(1)$ on the effective rank of the functional $f(\cdot)$..

Both components on the left-hand side of the first condition in \eqref{th-1-2} are needed.
For example, when $C_0\asymp 1$ and $\sigma\asymp 1$, the component $r^{1/2}(\bSigma)/n$
can be omitted when $V_1\asymp 1$ but needed when $V_1 = o(1)$.
Note that in a previous work \cite{koltchinskii2021}, asymptotic normality was proved under i.i.d. Gaussian assumption on $\bep_j$'s. Here, only independence is required. Especially, an interesting result of \cite{koltchinskii2021} is that when $V_1\asymp 1$ and $r(\bSigma) = n^{\alpha}$, the condition $s>1/(1-\alpha)$ is sharp for \eqref{th-1-3} in the Gaussian shift model. In Theorem \ref{th-1}, \eqref{th-1-2} matches this sharp threshold level on the smoothness index $s$ under the first and second moment conditions on the noise without additional distributional assumption.

\section{Estimation of non-smooth additive functions}

In this section we consider the estimation of additive functions of the form
\bel{f_0}
f(\btheta) = \frac{1}{d}\sum_{a=1}^d f_0(\theta_a),\quad \btheta = (\theta_1,\ldots,\theta_d)^\top,
\eel
for a given but possibly non-smooth function $f_0(\cdot)$,
based on independent observations $\bx_j = (x_{j,1},\ldots,x_{j,d})^\top$ with $\EE[\bx_j]=\btheta$.

We allow general $f_0$ satisfying a smoothness condition of small order,
including for example $f_0(x) =|x|^{p}$ with small $p > 0$.
Such functions are called non-smooth in the literature as their order of smoothness
is much smaller than the order of differentiability used in the analysis.

We assume the following conditions on the noise:
\bel{noise-cond}
&
\EE[\varepsilon_{j,a}\varepsilon_{j,b}]=0,\ \  1\le j\le n, 1\le a < b\le d,
\\ \nonumber
& \EE\big[|n^{-1}\sum_{j=1}^n \veps'_j\veps_{j,a}|^{2k}\big] \le \sigma_n^{2k}2^{k-1}k!,\ \ 1\le a \le d, 1\le k \le \lceil s\rceil,
\eel
where $\varepsilon_{j,a} = x_{j,a} - \theta_a$ is the noise for the
$a$-th component of the $j$-th observation, and
$\{\veps'_j, 1\le j\le n\}$ is a Rademacher sequence independent of $\{\bep_j, 1\le j\le n\}$.

We may set $\sigma_n=1$ so that the model matches the Gaussian shift model
where only a single instance of $\calN(\btheta,\bI_d)$ is observed.
Typically in this setting,
consistent estimation of a single $f_0(\theta_a)$ is infeasible
as there is only one $\calN(\theta_a,1)$ observation,
and the consistent estimation of $f(\btheta)$ is achieved through high-order
bias correction and noise averaging over data components $a=1,\ldots, d$.
Thus, for $\sigma_n \asymp 1$ and general $f_0$ we assume $d\to\infty$
and apply HODSE with $m\to\infty$ in \eqref{new-est-2}.
Two important features of our approach become more crucial when $m\to\infty$, namely
(a): the explicit description of the remainder term in Proposition \ref{prop-1};
and (b) sharp control of constant factors in Lemma \ref{lm-1}
for the analysis of the estimator \eqref{new-est-2}.
Of course, our approach also allows $\sigma_n\to 0$ and $\sigma_n\to \infty$.

The first line of \eqref{noise-cond} asserts that the elements $\veps_{j,a}$ of the noise vector $\bep_j$
are uncorrelated for each $j$ in addition to the independence of $\bep_1,\ldots,\bep_n$.
The second line of \eqref{noise-cond} can be viewed as a sub-Gaussian condition
which holds when $\EE[\varepsilon_{j,a}^{2k}]\le \EE[|\calN(0,n\sigma_n^2)|^{2k}]$ for all positive integers $k\le s$
due to the independence of $\bep_1,\ldots,\bep_n$.
For example, \eqref{noise-cond} holds when $\veps_{j,a}$ are i.i.d. with $\P\{\veps_{j,a}=\pm n^{1/2}\sigma_n\}=1/2$.
The i.i.d. $\varepsilon_{j,a}\sim \calN(0,n)$ assumption, which implies \eqref{noise-cond} with $\sigma_n=1$,
is equivalent to the standard Gaussian shift model due to the sufficiency of $\bar{\bx}\sim \calN(\btheta,\bI_d)$.
However, \eqref{noise-cond} is much weaker than the Gaussian shift model
as the components of $\bx_j$ are only required to be uncorrelated, and non-Gaussian data
and heteroscedasticity are allowed.

In \cite{cai2011} the authors proved that based on a single $\calN(\btheta,\bI_d)$ observation,
the minimax rate for the estimation of the length normalized $\ell_1$ norm $f(\btheta) = \|\btheta\|_1/d$ is
\bes
\inf_{\widehat{f}}\sup_{\btheta\in\R^d} \EE\Big[\big(\widehat{f} - f(\btheta)\big)^2\Big] \asymp (\log d)^{-1}.
\ees
We note that the plug-in estimator is inconsistent here as there is only one observation available
and $f_0(x) = |x|$ is not analytic. As the $\calN(\btheta,\bI_d)$ model is a special case of \eqref{noise-cond},
our results in this section will extend their upper bound to more general $f_0$ and
heteroscedastic non-Gaussian data.

Our idea is to first apply kernel smoothing to the function $f_0$,
\bel{kernel-smoothing}
\qquad && f_h(x) := (K_h*f_0)(x) = \int K_h(x-t)f_0(t)dt,\ \ {\rm with}~~K_h(x) := h^{-1}K(x/h),
\eel
with proper a kernel $K(x)$ and a bandwidth $h>0$, and then apply the proposed \eqref{new-est-2} to
estimate the individual $f_h(\theta_a)$, $a=1,\ldots,d$. This leads to the estimator
\bel{separable-est}
\widehat{f} := \frac{1}{d}\sum_{a=1}^d \bigg\{ f_h(\bar{x}_a)
+ \sum_{k=2}^m \frac{f_h^{(k)}(\bar{x}_a)\bar{u}_a^{(k)}}{k!}\bigg\}
\eel
with $\bar{x}_a = \sum_{j=1}^n x_{j,a}/n$ and
$\bar{u}_a^{(k)} = \{(n-k)!/n!\}\sum_{1\le j_1\neq\cdots\neq j_k\le n}\prod_{\ell=1}^k (x_{j_\ell,a} - \bar{x}_a)$.

Let $m=s-1\le n$ for some integer $s>2$. By Proposition \ref{prop-1} and \eqref{lm-1-1} of Lemma \ref{lm-1},
\bel{separable-expansion}
\widehat{f} = \frac{1}{d}\sum_{a=1}^d f_h(\theta_a)
+ \sum_{k=1}^m \frac{1}{d}\sum_{a=1}^d \frac{f_h^{(k)}(\theta_a)\bar{\veps}_a^{(k)}}{k!}
- \frac{1}{d}\sum_{a=1}^d \Rem_{m,a}
\eel
with completely degenerate $\bar{\veps}_a^{(k)} = \{(n-k)!/n!\}\sum_{1\le j_1\neq\cdots\neq j_k\le n}\prod_{\ell=1}^k \veps_{j_\ell,a}$ and
\bel{separable-rem}
|\Rem_{m,a}| \le \big\|f_h\big\|_{(s)}\sum_{k=0}^m
\frac{|\bar{\veps}_a|^{s-k}|\bar{\veps}_a^{(k)}|}{(s-k)!k!}
\eel
where $\bar{\veps}_a = n^{-1}\sum_{j=1}^n \veps_{j,a}$ is the average of the $a$-th noise component.

It can be seen from \eqref{separable-expansion} that
while the kernel smoothing \eqref{kernel-smoothing} introduces some bias,
correction of additional bias is achieved through the application of HODSE
in the estimation of individual $f_h(\theta_a)$, and de-noising is
achieved through the averaging in the degenerate $U$-variables $\bar{\veps}_a^{(k)}$
and over the d-coordinates and terms of different orders in the second term.

By \eqref{noise-cond},
$\bar{\veps}_a^{(k)}, 1\le a \le d, 1\le k\le m$,
are uncorrelated with each other and
\bel{separable-U-bd}
\EE\bigg[\bigg(\sum_{a=1}^d \sum_{k=1}^m \frac{f_h^{(k)}(\theta_a)\bar{\veps}_a^{(k)}}{d(k!)}\bigg)^2\bigg]
\le \sum_{a=1}^d \sum_{k=1}^m C_{k,n}^2 \frac{|f_h^{(k)}(\theta_a)|^2\sigma_n^{2k}}{d^2 k!}
\eel
with $C_{k,n}\le e^{(k-1)k/n}$.
By the second line of \eqref{noise-cond}, the remainder is bounded by
\bel{separable-Rem-bd}
\EE\bigg[\bigg( \frac{1}{d}\sum_{a=1}^d \Rem_{m,a}\bigg)^2\bigg]
\le C_{s,n}^2 \|f_h\|_{(s)}^2\sigma_n^{2s}2^{7s-1}/s!.
\eel
The benefit of the neat formulas \eqref{new-prop-1-1} and \eqref{new-prop-1-2} is evident
in this analysis in view of \eqref{separable-Rem-bd} in the case of $\sigma_n\asymp 1$.
It follows from \eqref{separable-expansion}, \eqref{separable-U-bd} and \eqref{separable-Rem-bd}
that for integers $s$ satisfying $s>2$ and $2(s-1)s/n\le\log 2$
the error of the estimator \eqref{separable-est} is bounded by
 \bel{separable-error-bd}
\quad && \big\|\widehat{f}-f(\btheta)\big\|_{L_2(\mathbb{P})} \le
\big\{\bias_{h,d}^2(\btheta) +\kappa_{s,h,n,d}(\btheta)\big\}^{1/2}
+ \|f_h\|_{(s)}(2^{7/2}\sigma_n)^s\big/\sqrt{s!}
\eel
for the $L_2(\mathbb{P})$ risk given by $\|\widehat{f}-f(\btheta)\|_{L_2(\mathbb{P})}^2 = \EE\big[|\widehat{f}-f(\btheta)|^2\big]$,
where
\bel{separable-bias-var}
\bias_{h,d}(\btheta) = \frac{1}{d}\sum_{a=1}^d (f_0-f_h)(\theta_a),\
\kappa_{t,h,n,d}(\btheta) = \sum_{k=1}^{\lceil t\rceil-1}
\frac{\|f_h^{(k)}(\btheta)\|_2^2\sigma_n^{2k}}{(d^2/2)k!}
\eel
are respectively the bias introduced by \eqref{kernel-smoothing}
and an upper bound for the variance in \eqref{separable-U-bd}.

The convergence rate and sharpness of \eqref{separable-error-bd} would heavily depend on
the smoothness properties of $f_0$ and $f_h$ through the choice of the kernel $K(\cdot)$ in \eqref{kernel-smoothing}.
The following theorem provides an explicit
error bound in an ``$\alpha$-smooth'' scenario, allowing arbitrarily small smoothness index $\alpha$ for the function $f_0(\cdot)$.

\begin{restatable}{theorem}{TheoremSeparable}\label{th-separable}
Let $s$ be an integer satisfying $2<s \le \sqrt{(n/2)\log 2}$.
Let $f(\btheta)$ be as in \eqref{f_0} and $\widehat{f}$ as in \eqref{separable-est} with $m=s-1$.
If \eqref{noise-cond} holds, then \eqref{separable-error-bd} holds.
If \eqref{noise-cond} holds and
\bel{f-cond-alpha}
\qquad &&
\|f_0^{(k)}-f_h^{(k)}\|_{L_\infty}\le \eta_\alpha(h)/h^{k},\ 0\le k < \alpha,\ \
\|f_h^{(k)}\|_{L_\infty} \le \eta_\alpha(h)/h^{k},\ \alpha  \le k \le s,\ \
\eel
for certain $\alpha>0$ and $\eta_\alpha(h)>0$
for the $h$ and $f_h$ in \eqref{kernel-smoothing} with $(\sigma_n/h)^2 \le s/(2^7e)$, then
\bel{th-separable-1}
\qquad \big\|\widehat{f}-f(\btheta)\big\|_{L_2(\mathbb{P})}
&\le& \big|\bias_{h,d}(\btheta)\big|
+ \eta_\alpha(h)\big(s^{-1/4}+\sqrt{2/d}e^{(\sigma_n/h)^2/2}\big)
+ \kappa^{1/2}_{\alpha,0,n,d}(\btheta)
\cr &\le& \eta_\alpha(h)\big(1+ s^{-1/4}+\sqrt{2/d}e^{(\sigma_n/h)^2/2}\big)
+ \kappa^{1/2}_{\alpha,0,n,d}(\btheta)
\eel
where $\bias_{h,d}(\btheta)$ and $\kappa_{\alpha,h,n,d}(\btheta)$
are as defined in \eqref{separable-bias-var}.
\end{restatable}

\begin{corollary}\label{cor-5-1}
Suppose \eqref{noise-cond} and \eqref{f-cond-alpha} hold with $\eta_\alpha(h) =C_\alpha h^\alpha$
for some constants $\alpha \in (0,1]$ and $C_\alpha<\infty$.
When $(\sigma_n/h)^2 = \log(d/\log d)$ in \eqref{th-separable-1},
\bel{th-separable-2}
\quad \big\|\widehat{f}-f(\btheta)\big\|_{L_2(\mathbb{P})}
&\le& \big|\bias_{h,d}(\btheta)\big|+ \big(s^{-1/4} +\sqrt{2/\log d} \big)
C_\alpha \sigma_n^\alpha \big/(\log(d/\log d)\big)^{\alpha/2}
\cr &\le& \big(C_\alpha + o(1)\big)\sigma_n^\alpha/(\log d)^{\alpha/2}.
\eel
\end{corollary}

We notice that when $\veps_{j,a}$ are i.i.d. $\calN(0,\sigma^2)$, we have $\sigma_n=\sigma/\sqrt{n}$,
so that for $\alpha=1$ the convergence rate in \eqref{th-separable-2}
is faster than the usual parametric rate
due to the noise averaging over components $a=1,\ldots,d$.
At high noise level $\sigma_n\asymp 1$, consistent estimation of $f(\btheta)$ is still feasible
but the convergence rate is logarithmic for general $f_0$.

Interestingly, in \eqref{th-separable-2}, the condition on the order of expansion $m=s-1$
in the construction of the estimator \eqref{separable-est},
$2^7e\log(d/\log d) \le s \le \sqrt{(n/2)\log 2}$,
depends on $n$ and $d$ only, not on the noise level $\sigma_n$,
while the choice of the bandwidth in \eqref{kernel-smoothing},
$h = \sigma_n/\sqrt{\log(d/\log d)}$, is proportional to the noise level.
We also note from the theorem that the risk is typically
dominated by the bias in the kernel smoothing step.
This phenomenon also presents in some other problems with logarithmic convergence rates
such as deconvolution \cite{zhang1990fourier,fan1991optimal}.

Theorem \ref{th-separable} is proved in Section \ref{section: proofs} along with detailed derivations of
\eqref{separable-U-bd} and \eqref{separable-Rem-bd}. In the following two subsections, we verify condition \eqref{f-cond-alpha} for specific $f_0$
and choice of the kernel $K(x)$ in \eqref{kernel-smoothing}.
We note that while our general solution achieves the existing optimal minimax rate in
these two examples, the existing estimator uses polynomial approximation specifically
constructed for the individual $f_0$.

\subsection{Estimation of the $\ell_1$-norm}
For $f_0(x)=|x|$, the function is 1-Lipschitz and $f_h^{(2)}(x) = 2K_h(x) = 2h^{-1}K(x/h)$
in \eqref{kernel-smoothing}. Thus,
condition \eqref{f-cond-alpha} holds with $\alpha=1$ and $\eta_1(h)=C_1 h$
for all integers $s > 2$ when the kernel in \eqref{kernel-smoothing} satisfies
\bel{K-cond-ell_1}
\quad && \int K(x)dx =1,\ \int |x^\ell K(x)|dx \le C_1,\ 2\|K^{(k-2)}\|_{L_\infty} \le C_1,
\eel
for $\ell\in\{0,1\}$ and $2\le k\le s$.  When $K(\cdot)$ is the Fourier inversion of
a twice continuously differentiable function $Q(\cdot)$ with support $[-1,1]$,
these conditions on $K(\cdot)$ hold when
\bel{Q-cond-ell_1}
\qquad && Q(0) = \frac{1}{\sqrt{2\pi}},\ \supp(Q)=[-1,1],\ \|Q\|_{L_1}\le \frac{C_1}{\sqrt{2/\pi}},\
\max_{k=0,2}\big\|Q^{(k)}\big\|_{L_2}\le \frac{C_1}{\sqrt{2\pi}}.
\eel
The verification of the above claims is elementary but will be included
in the proof of the following theorem for completeness.

\begin{restatable}{theorem}{TheoremEllOne}\label{th-ell_1}
Let $s =\lceil 2^7 e\, \log(d/\log d)\rceil$ and $m=s-1$.
Suppose \eqref{noise-cond} holds and
$s \le \sqrt{(n/2)\log 2}$.
Let $f_h$ be as in \eqref{kernel-smoothing} with $f_0(x)=|x|$, $h = \sigma_n/\sqrt{\log(d/\log d)}$ and
a kernel function $K(\cdot)$ satisfying \eqref{K-cond-ell_1}.
Let $f(\btheta)=\|\btheta\|_1/d$ and $\widehat{f}$ be as in \eqref{separable-est}. Then,
\bes
\EE\big[|\widehat{f}-f(\btheta)|^2\big] \le \big(1+\sqrt{2/\log d} + s^{-1/4}\big)^2
C_1^2\sigma_n^2\big/\log(d/\log d).
\ees
Moreover, \eqref{K-cond-ell_1} holds when $K(\cdot)$ is
the Fourier inversion of a function $Q$ satisfying \eqref{Q-cond-ell_1}.
\end{restatable}

The above theorem demonstrates that at the high-noise level $\sigma_n=1$ (e.g. $\Var(\veps_{j,a})=n$),
the degenerate statistical expansion \eqref{new-est-2} in high-order $s\asymp \sqrt{\log d}$
yields the rate optimal minimax upper bound in \cite{cai2011}.
The second conclusion of the theorem also holds when $K(x)$ is taken as the real part
of the Fourier inversion of a function $Q$ satisfying \eqref{Q-cond-ell_1}.

\subsection{Estimation of the $\ell_{p}$-norm}
For $f_0(x)=|x|^{p}$ with $p\in (0,1)$,
Theorem \ref{th-separable} is still applicable when
$K(\cdot)$ is the inverse Fourier transform of a function $Q(\cdot)$ satisfying
\eqref{Q-cond-ell_1}.

\begin{restatable}{theorem}{TheoremEllAlpha}\label{th-ell_alpha}
Let $f(\btheta)=\|\btheta\|_{p}^{p}/d$ with $0<p < 1$ and $f_0(x)=|x|^{p}$ in \eqref{f_0} .
Let $f_h$ be as in \eqref{kernel-smoothing} with $h = \sigma_n/\sqrt{\log(d/\log d)}$ and a kernel
$K(\cdot)$ being the inverse Fourier transform of $Q(\cdot)$ satisfying \eqref{Q-cond-ell_1}.
Let $s =\lceil 2^7 e\,\log(d/\log d) \rceil$ and $\widehat{f}$ be the estimator in \eqref{separable-est} with $m=s-1$.
Suppose \eqref{noise-cond} holds and
$s \le \sqrt{(n/2)\log 2}$. Then,
\bes
\EE\big[|\widehat{f}-f(\btheta)|^2\big] \le \big(1+\sqrt{2/\log d} + s^{-1/4}\big)^2
C_1^2\big(\sigma_n^2\big/\log(d/\log d)\big)^p.
\ees
\end{restatable}

Again, the degenerate statistical expansion \eqref{new-est-2} in high-order $s\asymp \sqrt{\log d}$
yields the upper bound $O(\sigma_n^{2 p}/(\log d)^{p})$ on MSE which was shown to be minimax optimal
by \cite{collier2020} in the dense region.  We note that the estimator in \cite{collier2020} is implicitly defined through
the optimal $(2m)$-degree polynomial approximation of $|x|^p$ under the $L_\infty$ norm on $[-1,1]$
with integer $m\asymp \log d$.

\section{Proofs}
\label{section: proofs}

\begin{proof}[Proof of Proposition \ref{prop-1}]
An order $k$ tensor $T^{(k)}$ is symmetric if
\bes
\langle T^{(k)},\bv_1\otimes\cdots\otimes \bv_k\rangle_k
= \langle T^{(k)},\bv_{j_1}\otimes\cdots\otimes \bv_{j_k}\rangle_k
\ees
for all permutations of indices $\{j_1,\ldots, j_k\}=\{1,\ldots,k\}$.
 By algebra and the definition of $\bar{\bu}^{(k)}$,
 for any symmetric tensor $T^{(k)}$
\bel{pf-prop-1-1}
\big\langle T^{(k)},\bar{\bu}^{(k)}\big\rangle_k
&=& \sum_{1\le j_1\neq \cdots \neq j_k\le n}
\frac{\langle T^{(k)},(\bx_{j_1}-\bar{\bx})\otimes \cdots\otimes(\bx_{j_k} - \bar{\bx})\rangle_k}{n(n-1)\cdots(n-k+1)}
\cr &=& \sum_{1\le j_1\neq \cdots \neq j_k\le n}
\frac{\langle T^{(k)},
(\bep_{j_1}-\bar{\bep})\otimes \cdots\otimes(\bep_{j_k} - \bar{\bep})\rangle_k}{n(n-1)\cdots(n-k+1)}
\\ \nonumber &=& \sum_{j=0}^k {k\choose j}(-1)^j \langle T^{(k)}, \bar{\bep}^{\otimes j}\bar{\bep}^{(k-j)}\big\rangle_k.
\eel

Let $h_k(y) = \langle f^{(k)}(\bar{\bx} + y(\btheta - \bar{\bx})),\bar{\bep}^{(k)}\rangle_k$.
We have $h_k(1) =\langle f^{(k)}(\btheta),\bar{\bep}^{(k)}\rangle_k$  and
$h_k(0) =\langle f^{(k)}(\bar{\bx}),\bar{\bep}^{(k)}\rangle_k$.
Let $(J^\alpha h)(y) = \int_0^y h(t)(y-t)^{\alpha-1}dt/\Gamma(\alpha)$ be the Riemann–Liouville integral,
$J^\alpha h =(J^\alpha h)(1)$,
and $\Delta_h^{(j)}(t) = h^{(j)}(t)-h^{(j)}(0)$ for any function $h(t)$.
We write the Taylor expansion of $h_k(y)$ as
\bes
h_k(y) = \sum_{j=0}^{m-k-1} \frac{h_k^{(j)}(0)y^j}{j!} +(J^{m-k}h_k^{(m-k)})(y)
= \sum_{j=0}^{m-k} \frac{h_k^{(j)}(0)y^j}{j!} + (J^{m-k}\Delta_{h_k}^{(m-k)})(y).
\ees
As $\btheta - \bar{\bx} = - \bar{\bep}$,
\bes
h_k^{(j)}(y) = \langle f^{(k+j)}\big(\bar{\bx} + y(\btheta - \bar{\bx})\big),
(-\bar{\bep})^{\otimes j}\otimes \bar{\bep}^{(k)}\rangle_{k+j}.
\ees
Thus, taking $y=1$ in the expression for $h_k(y)$, we find that the Taylor expansion for $h_k(1)$ is
\bes
\langle f^{(k)}(\btheta),\bar{\bep}^{(k)}\rangle_k
&=& \sum_{j=0}^{m-k} \frac{h_k^{(j)}(0)}{j!}
+ J^{m-k}\Delta_{h_k}^{(m-k)}
\cr &=&
\sum_{j=0}^{m-k} \frac{\langle f^{(k+j)}\big(\bar{\bx})\big),
\bar{\bep}^{\otimes j}\otimes \bar{\bep}^{(k)}\rangle_{k+j}}{(-1)^j j!}
+ \frac{\langle J^{m-k}\Delta^{(m)},
\bar{\bep}^{\otimes (m-k)}\otimes \bar{\bep}^{(k)}\rangle_{m}}{(-1)^{m-k}}.
\ees
Summing the above expression over $k=0,\ldots,m$, we have
\bel{pf-prop-1-2}
\sum_{k=0}^m \frac{\langle f^{(k)}(\btheta),\bar{\bep}^{(k)}\rangle_k}{k!}
= \sum_{k=0}^m \sum_{j = 0}^{m-k}
\frac{\langle f^{(k+j)}(\bar{\bx}),\bar{\bep}^{\otimes j}\otimes \bar{\bep}^{(k)}\rangle_{k_j}}{(-1)^j j! k!}
+ \Rem_m
\eel
with the remainder term
\bes
\Rem_m
= \sum_{k=0}^{m} \frac{\langle J^{m-k}\Delta^{(m)},
\bar{\bep}^{\otimes (m-k)}\otimes \bar{\bep}^{(k)}\rangle_m}{(-1)^{m-k} k!}
\ees
in \eqref{new-prop-1-2}.
Let $\ell+j=k$. As $f^{(k)}(\bar{\bx})$ are symmetric tensors,
\bes
\widehat{f}
&=& \sum_{k=0}^m \frac{\langle f^{(k)}(\bar{\bx}),\bar{\bu}^{(k)}\rangle_k}{k!}
\cr &=& \sum_{0\le j\le k\le m} {k\choose j}
\frac{\langle f^{(k)}(\bar{\bx}),\bar{\bep}^{\otimes j}\otimes \bar{\bep}^{(k-j)}\rangle_k}{(-1)^j k!}
\cr &=& \sum_{\ell=0}^m \sum_{j = 0}^{m-\ell}
\frac{\langle f^{(\ell+j)}(\bar{\bx}),\bar{\bep}^{\otimes j}\otimes \bar{\bep}^{(\ell)}\rangle_{\ell+j}}{(-1)^j j! \ell!}
\ees
by the formula for $\bar{\bu}^{(k)}$ in \eqref{pf-prop-1-1}. This and \eqref{pf-prop-1-2} yield \eqref{new-prop-1-1}.
\end{proof}

\begin{proof}[Proof of Lemma \ref{lm-1}]
By the definition of the norm $\|T^{(m)}\|_{m-k,k}$,
\begin{equation*}
\label{lm-1-p-1}
  |\Rem_m| \leq
\sum_{k=0}^{m} \Big(\max_{0<t\leq 1} \frac{\|\Delta^{(m)}(t)\|_{m-k,k}}{(t\|\bar{\bep}\|)^{s-m}} \Big) \int_0^1 \frac{t^{s-m}
\|\bar{\bep}\|^{(s-k)} \|\bar{\bep}^{(k)}\|_{HS}}{(m-k-1)! k!}(1-t)^{m-k-1}dt.
\end{equation*}
Then inequality \eqref{lm-1-1} follows directly by applying the following fact to \eqref{lm-1-p-1}:
\bes
\int_0^1 \frac{t^{s-m}(1-t)^{m-k-1}}{(m-k-1)!} dt = \frac{\Gamma(s-m+1)\Gamma(m-k)}{\Gamma(s-k+1)(m-k-1)!}
= \frac{\Gamma(s-m+1)}{\Gamma(s-k+1)}\le  \frac{1}{\Gamma(s-k+1)}.
\ees
The inequality follows from the convexity of $\Gamma(x)$ in $[1,2]$ and $\Gamma(2)=\Gamma(1)=1$.

To prove \eqref{lm-1-2}, we need to write $\bar{\bep}^{\otimes (k-\ell)}\bar{\bep}^{(k)}$
as a sum of rank-one tensors before we can apply the spectrum norm on $\Delta^{(m)}(t)$.
To this end we observe that for any order $k$ symmetric tensor $T^{(k)}$
and the completely degenerate tensor $\bar{\bep}^{(k)}$ in \eqref{new-eps-(k)}
\bes
&& (n^k/C_{k,n})\langle T^{(k)},\bar{\bep}^{(k)}\rangle_k
\cr &=& \bigg\langle T^{(k)},\sum_{1\le j_1\neq \cdots \neq j_k\le n}
\bep_{j_1}\otimes \cdots\otimes \bep_{j_k} \bigg\rangle_k
\cr &=& \bigg\langle T^{(k)}, \sum_{1\le j_1\neq \ldots \neq j_{k-1}\le n}
\bep_{j_1}\otimes \cdots\otimes \bep_{j_{k-1}}\otimes \bigg(n\bar{\bep} - \sum_{a=1}^{k-1}\bep_{j_a}\bigg)
\bigg\rangle_k
\cr &=& \bigg\langle T^{(k)}, \big(n\bar{\bep}\big)\otimes \sum_{1\le j_1\neq \ldots \neq j_{k-1}\le n}
\bep_{j_1}\otimes \cdots\otimes \bep_{j_{k-1}}\bigg\rangle_k
\cr && - (k-1)\bigg\langle T^{(k)}, \big(n\bar{\bep}\big)\otimes \sum_{1\le j_1\neq \ldots \neq j_{k-1}\le n}
\bep_{j_1}^{\otimes 2}\otimes \cdots\otimes \bep_{j_{k-1}}\bigg\rangle_k.
\ees
This turns a summation of $k$ indices into $k$ summations of $(k-1)$ indices, counting multiplicity.
We shall repeat this process until
we write the above into a sum of terms of the following format
\bes
(-1)^{(k_1-1)+\cdots+(k_b-1)}\bigg\langle T^{(k)},
\big(n\bar{\bep}\big)^{\otimes(k-\ell)} \otimes \sum_{1\le j_1\neq \ldots \neq j_b\le n}
\bep_{j_1}^{\otimes k_1}\otimes \cdots\otimes \bep_{j_b}^{\otimes k_b}\bigg\rangle_k
\ees
for some integer $b\in [0,k/2]$ and $2\le k_1\le \cdots \le k_b$ satisfying $\sum_{a=1}^b k_a=\ell$.
By induction, the sum of the multiplicities of such terms, say $C^{(k)}_{k_1,\ldots,k_b}$
is bounded by $k!$. This means
\bes
&& (n^k/C_{k,n})\langle T^{(k)},\bar{\bep}^{(k)}\rangle_k
\cr &=&
\sum_{b\ge 0, 2\le k_1\le\cdots\le k_b \atop k_1+\ldots + k_b = \ell \le k}
\frac{C^{(k)}_{k_1,\ldots,k_b}}{(-1)^{\ell-b}}
\bigg\langle T^{(k)}, \big(n\bar{\bep}\big)^{\otimes(k-\ell)} \otimes \sum_{1\le j_1\neq \ldots \neq j_b\le n}
\bep_{j_1}^{\otimes k_1}\otimes \cdots\otimes \bep_{j_b}^{\otimes k_b}\bigg\rangle_k
\ees
where $C^{(k)}_{k_1,\ldots,k_b}$ are positive integers satisfying
\bes
r_k = \sum_{b\ge 0, 2\le k_1\le\cdots\le k_b \atop k_1+\ldots + k_b = \ell \le k}
C^{(k)}_{k_1,\ldots,k_b}\le k!.
\ees
It follows by simple algebra that
\bes
&& \langle T^{(k)},\bar{\bep}^{(k)}\rangle_k
\cr &=&
\sum_{b\ge 0, 2\le k_1\le\cdots\le k_b \atop k_1+\ldots + k_b = \ell \le k}
\frac{C_{k,n}C^{(k)}_{k_1,\ldots,k_b}}{C_{b,n}(-1)^{\ell-b}}
\bigg\langle T^{(k)}, \sum_{1\le j_1\neq \ldots \neq j_b\le n}
\frac{\bep_{j_1}^{\otimes k_1}\otimes \cdots\otimes \bep_{j_b}^{\otimes k_b}\otimes \bar{\bep}^{\otimes(k-\ell)}}
{n(n-1)\cdots(n-b+1)n^{\ell-b}} \bigg\rangle_k.
\ees
Applying the above identity to each term in the remainder formula \eqref{new-prop-1-2} and the spectrum norm,
we find that
\bes
&& \Big|\langle \Delta^{(m)}(t),\bar{\bep}^{\otimes (m-k)}\otimes \bar{\bep}^{(k)}\rangle_m\Big|
\cr &\le& \|\Delta^{(m)}(t)\|_{\rm S}
\sum_{b\ge 0, 2\le k_1\le\cdots\le k_b \atop k_1+\ldots + k_b = \ell \le k}
\frac{C_{k,n}C^{(k)}_{k_1,\ldots,k_b}}{C_{b,n}}
\sum_{1\le j_1\neq \ldots \neq j_b\le n}
\frac{\|\bep_{j_1}\|^{k_1}\cdots\|\bep_{j_b}\|^{k_b}\|\bar{\bep}\|^{m-\ell}}{n(n-1)\cdots(n-b+1)n^{\ell-b}}.
\ees
This gives \eqref{lm-1-2} as in \eqref{lm-1-p-1}.

Finally we provide an upper bound for $C_{k,n}$ by a variation of
the Stirling's formula. Let $c_n = n! e^n/(n+1/2)^{n+1/2}$ and $x_n=1/(n+1)$.
We have
\bes
\frac{c_{n+1}}{c_n}
= \frac{(n+1)e (n+1/2)^{n+1/2}}{(n+3/2)^{n+3/2}}
= \frac{e (1-x_n/2)^{1/x_n-1/2}}{(1+x_n/2)^{1/x_n+1/2}}
= e\bigg(\frac{(1-x_n/2)^{1-x_n/2}}{(1+x_n/2)^{1+x_n/2}}\bigg)^{1/x_n}.
\ees
Due to $-\log(1-t/2) - \log(1+t/2)\ge 0$
\bes
&& \frac{1}{x}\bigg(\bigg(1-\frac{x}{2}\bigg)\log\bigg(1-\frac{x}{2}\bigg) -
\bigg(1+\frac{x}{2}\bigg)\log\bigg(1+\frac{x}{2}\bigg)\bigg)
\cr &=& \frac{1}{x}\int_0^x \bigg(\frac{-1}{2} - \frac{1}{2}\log\bigg(1-\frac{t}{2}\bigg)
- \frac{1}{2} - \frac{1}{2}\log\bigg(1+\frac{t}{2}\bigg)\bigg)dt \ge -1,
\ees
so that $c_{n+1}\ge c_n$.  It follows that
\bes
C_{k,n}
&=& \frac{n^{k-1}(n-k)!}{(n-1)!}
\cr &\le& \frac{n^{k-1}e^{k-n}(n-k+1/2)^{n-k+1/2}}{e^{1-n} (n-1/2)^{n-1/2}}
\cr &\le & e^{k-1}\bigg(1+\frac{1/2}{n-1/2}\bigg)^{k-1}\bigg(1 - \frac{k-1}{n-1/2}\bigg)^{n-k+1/2}
\cr &\le & \exp\bigg(k-1 + \frac{(k-1)/2}{n-1/2} - \frac{(k-1)(n-k+1/2)}{n-1/2}\bigg)
\cr & = & \exp\big((k-1)(k-1+1/2)/(n-1/2)\big).
\ees
This gives $C_{k,n}\le e^{(k-1)(2k-1)/(2n-1)} \le e^{(k-1)k/n}$.
\end{proof}

\begin{proof}[Proof of Theorem \ref{TheoBias}]
By definition $\|\Delta^{(m)}(t)\|_{\rm S}\le \|f\|_{(s)}(t\|\bar{\bep}\|)^{s-m}$.
By \eqref{lm-1-2},
\bes
\|\Rem_m\|_{L_p(\P)}
\le C_{m,n}'\|f\|_{(s)}
\max_{k_1+\ldots + k_b = \ell \le m \atop k_a\ge 2\,\forall a; b\ge 0}
\bigg\|\sum_{j_1\neq \ldots \neq j_b}
\frac{\|\bep_{j_1}\|^{k_1}\cdots\|\bep_{j_b}\|^{k_b}\|\bar{\bep}\|^{s-\ell}}
{n^{k_1+\cdots+k_b}}\bigg\|_{L_p(\P)}
\ees
where $C_{m,n}'$ is defined and bounded by
\bes
C_{m,n}' =
\sum_{k_1+\ldots + k_b = \ell \le k\le m \atop k_a\ge 2\,\forall a; b\ge 0}
\frac{C_{b,n}(C_{k,n}/C_{b,n})C^{(k)}_{k_1,\ldots,k_b}}{(m-k)! k!} \le \sum_{k=0}^m \frac{C_{k,n}}{(m-k)!} = C^*_{m,n}.
\ees
Let $\bw = (\|\bep_1\|/n,\ldots,\|\bep_n\|/n)^\top$.
Due to $\|\bw\|_k\le \|\bw\|_2$ for $k\ge 2$,
\bes
\|\Rem_m\|_{L_p(\P)}
&\le& C_{m,n}'\|f\|_{(s)}
\max_{k_1+\ldots + k_b = \ell \le m \atop k_a\ge 2\,\forall a; b\ge 0}
\bigg\|\|\bar{\bep}\|^{s-\ell}\prod_{a=1}^b \|\bw\|_{k_a}^{k_a}\bigg\|_{L_p(\P)}
\cr &\le& C_{m,n}'\|f\|_{(s)}
\max_{0\le \ell \le m}
\big\|\|\bar{\bep}\|^{s-\ell}\|\bw\|_2^{\ell}\big\|_{L_p(\P)}.
\ees
This gives \eqref{TheoBias-1} with an application of H\"older's inequality,
\bes
\|\Rem_m\|_{L_p(\P)}
&\le& C_{m,n}'\|f\|_{(s)}
\max\Big\{\|\bar{\bep}\|_{L_{ps}}^{s}, \big\|\|\bw\|_{2}\big\|_{L_{ps}(\P)}^s\Big\}
\cr & = & C_{m,n}^*\|f\|_{(s)}
\max\bigg\{\|\bar{\bep}\|_{L_{ps}}^{s}, \bigg\|\sum_{j=1}^n \frac{\|\bep_j\|^2}{n^2}\bigg\|^{s/2}_{L_{ps/2}(\P)}\bigg\}.
\ees

For independent $\bep_j$, $|\EE[\widehat{f}]-f(\btheta)|\le \|\Rem_m\|_{L_p(\P)}$
and Rosenthal's inequality applies:
\bes
\big\|\bar{\bep}\|_{L_{ps}(\P)}^{ps}
\le C'''_{ps} \bigg(\big\|\bar{\bep}\|_{L_2(\P)}^{ps}
+ \sum_{j=1}^n \frac{\|\bep_j\|_{L_{ps}(\P)}^{ps}}{n^{ps}}\bigg)
= C'''_{ps} \big(\big\|\bw\|_{L_2(\P)}^{ps} + \|\bw\|_{L_{ps}(\P)}^{ps}\big)
\ees
with $C'''_s = (1+s/\log(s/2))^s$ for $s\ge 4$ (\cite{kwapien1991}, Theorem 3.3), and
\bes
&& \bigg\| \bigg(\sum_{j=1}^n \frac{\|\bep_j\|^2}{n^2}\bigg) - \|\bar{\bep}\|_{L_2(\P)}^2\bigg\|^{ps/2}_{L_{ps/2}(\P)}
\cr &\le& C'''_{ps/2} \bigg(\bigg(\sum_{j=1}^n \frac{\|\bep_j\|_{L_4(\P)}^4}{n^4}\bigg)^{ps/4}
+ \sum_{j=1}^n \frac{\|\bep_j\|_{L_{ps}(\P)}^{ps}}{n^{ps}}\bigg)
\cr & = & C'''_{ps/2}  \big(\big\|\bw\|_{L_4(\P)}^{ps} + \|\bw\|_{L_{ps}(\P)}^{ps}\big)
\ees
with $C'''_{ps/2}\le C''_{ps}$.
Because $\big\|\bw\|_{L_4(\P)} \le \max\{\big\|\bw\|_{L_2(\P)}, \big\|\bw\|_{L_4(\P)}\}$,
the second inequality in \eqref{TheoBias-2} follows by inserting the above inequalities into
\eqref{TheoBias-1}.
\end{proof}

\begin{proof}[Proof of Theorem \ref{th-risk-bd}]
With the degenerate $U$-tensors $S_k$ in \eqref{S_k}, we write
\bes
\widehat{f} - f(\btheta) = \sum_{k=1}^m \frac{S_k}{k!} - \Rem_m.
\ees
We first bound the second moment of $S_k$ in the sum as the leading term.
For $k=1$,
\bes
\EE\big[S_1^2\big]
= \EE\big[\big\langle f^{(1)}(\btheta),\bep_1\big\rangle_1^2\big]\big/n
= V_1/n \le C_0\sigma^2/n.
\ees
For $k=2,\ldots,m$, $\EE\big[\big\langle f^{(k)}(\btheta),\bep_{j_1}\otimes \cdots\otimes \bep_{j_k}\big\rangle_k
\big\langle f^{(k)}(\btheta),\bep_{j'_1}\otimes \cdots\otimes \bep_{j'_k}\big\rangle_k\big]=0$ when
$j_1<\cdots<j_k$, $j_1'<\cdots<j_k'$ and $j_\ell \neq j'_\ell$ for some $\ell$. Thus, in the i.i.d. case
\bes
\EE\big[S_k^2\big]
= \frac{\EE\big[\big\langle f^{(k)}(\btheta),\bep_1\otimes \cdots\otimes \bep_k\big\rangle_k^2\big]}
{n(n-1)\cdots(n-k+1)/k!}
= \frac{C_{k,n}k!}{n^k}V_k,
\ees
which gives \eqref{var-S_k}. In the i.n.i.d. case,
$\bSigma=n^{-1}\sum_{j=1}^n\bSigma_j$ with $\bSigma_j = \EE[\bep_j\otimes\bep_j]$, so that
\bes
\EE\big[S_k^2\big]
&=& \bigg(\frac{C_{k,n}k!}{n^{k}}\bigg)^2 \sum_{1\le j_1<\cdots < j_k\le n}
\big\langle f^{(k)}(\btheta),f^{(k)}(\btheta)\times_1\bSigma_{j_1}\times_2  \cdots \times_k \bSigma_{j_k}\big\rangle_k
\cr &\le & \frac{C_{k,n}^2k!}{n^{2k}} \sum_{1\le j_1, \ldots, j_k\le n}
\big\langle f^{(k)}(\btheta),f^{(k)}(\btheta)\times_1\bSigma_{j_1}\times_2  \cdots \times_k \bSigma_{j_k}\big\rangle_k
\cr & = & (C_{k,n}^2k!/n^k)V_k,
\ees
which gives \eqref{var-S_k-inid}. Moreover, because $\EE[S_{k_1}S_{k_2}]=0$ for $1\le k_1<k_2\le m$,
\bes
\EE\bigg[\bigg(\sum_{k=1}^m S_k/k!\bigg)^2\bigg]
\le \sum_{k=1}^m \frac{C_{k,n}^2}{n^{k}k!}V_k.
\ees
This and \eqref{TheoBias-2} yield \eqref{th-risk-bd-1}.

Next we verify the upper bound \eqref{var-S_k-bd} for $V_k$.
For $k=2$,
\bes
V_2 &=& \big\langle f^{(2)}(\btheta),f^{(2)}(\btheta)\times_1 \bSigma \times_2\bSigma \big\rangle_2
\cr &=& \trace\Big(\big(f^{(2)}(\btheta)\times_1\bSigma^{1/2}\times_2\bSigma^{1/2}\big)^{\otimes 2}\Big)
\cr &\le &\big\| f^{(2)}(\btheta)\times_1\bSigma^{1/2}\times_2\bSigma^{1/2}\big\|_{\rm S}
\trace\big(f^{(2)}(\btheta)\times_1\bSigma^{1/2}\times_2\bSigma^{1/2}\big)
\cr &\le& \|f^{(2)}(\btheta)\|_{\rm S}\sigma^2\|f^{(2)}(\btheta)\|_{\rm S}\sigma^2r.
\ees
For $k>2$, we write $\bSigma = \sum_\ell\lam_\ell \bv_\ell\otimes \bv_\ell$ as the eigenvalue decomposition and
use the bound
\bes
V_k &=& \big\langle f^{(k)}(\btheta),f^{(k)}(\btheta)\times_1\bSigma\cdots\times_k\bSigma \big\rangle_k
\cr &=&\sum_{\ell_3,\ldots,\ell_k}\big\langle f^{(k)}(\btheta),f^{(k)}(\btheta)\times_1\bSigma\times_2\bSigma
\times_{j=3}^k  \big(\lam_{\ell_j}\bv_{\ell_j}^{\otimes 2}\big)\big\rangle_k
\cr &=&\sum_{\ell_3,\ldots,\ell_k}\lam_{\ell_3}\cdots\lam_{\ell_k}
\big\langle \big(f^{(k)}(\btheta)\times_{j=3}^k  \bv_{\ell_j}\big),\big(f^{(k)}(\btheta)\times_{j=3}^k  \bv_{\ell_j}\big)
\times_1\bSigma\times_2\bSigma \big)\big\rangle_2
\cr & \leq &\sum_{\ell_3,\ldots,\ell_k}\lam_{\ell_3}\cdots\lam_{\ell_k}
\big\|f^{(k)}(\btheta)\times_{j=3}^k  \bv_{\ell_j}\big\|_{\rm S}^2 \sigma^4r
\cr &\le &\sum_{\ell_3,\ldots,\ell_k}\lam_{\ell_3}\cdots\lam_{\ell_k}\|f^{(k)}(\btheta)\|_{\rm S}^2 \sigma^4 r
\cr &=& \|f^{(k)}(\btheta)\|_{\rm S}^2\sigma^{2k} r^{k-1}.
\ees
Thus, \eqref{var-S_k-bd} holds.

Inserting \eqref{var-S_k-bd} and bounds $\max_{2\le k <s} \|f^{(k)}(\btheta)\|_{\rm S}\le C_0$ and $\|f\|_{(s),\btheta}\le C_0$
into \eqref{th-risk-bd-1}, we obtain \eqref{th-risk-bd-2} because
$\max\big\{\sigma^2r^{1/2}/n,\cdots, \sigma^{m}r^{(m-1)/2}/n^{m/2}, \sigma^s(r/n)^{s/2}\big\}$ is attained at
$\max\big\{\sigma^2r^{1/2}/n,\sigma^s(r/n)^{s/2}\big\}$.
\end{proof}

\begin{proof}[Proof of Theorem \ref{th-1}]
 We shall keep in mind that all quantities, including $\{\scrH, f, s, \sigma, r, C_0\}$,
are allowed to depend on $n$, so that $O(1)$ here is uniformly bounded by numerical constants.
When $s^2\le n$, we have $C_{k,n}\le C_{m,n} \le e^{(m-1)m/n}\le e$
and $C^*_{m,n}\le e^{2}$ in Theorem \ref{TheoBias}.

We first prove that the remainder term is of smaller order than
the standard deviation $(V_1/n)^{1/2}$ of the linear term, for the $V_1$
defined in \eqref{V_k}. Let $\bw = (\|\bep_1\|/n,\ldots,\|\bep_n\|/n)^\top$.
By the definition of the noise level $\sigma$ and effective rank $r$ in \eqref{effective-rank},
$\EE[\|\bar{\bep}\|^2] = \EE[\|\bw\|_2^2] = \sigma^2r/n$.
Thus, by Chebyshev's inequality, in a certain event $\Omega_M$ with $\P\{\Omega_M\}\ge 1- 2/M^2$
\bel{M-bd}
\|\bar{\bep}\|^2 \le M^2\sigma^2r/n,\quad  \|\bw\|_2^2 \le M^2\sigma^2r/n.
\eel
By \eqref{M-bd} and \eqref{lm-1-2} the remainder term is bounded in $\Omega_M$ by
\bes
|\Rem_m|
&\le & 2\|f\|_{(s),\btheta}
\sum_{k_1+\ldots + k_b = \ell \le k\le m \atop k_a\ge 2\,\forall a; b\ge 0}
\frac{(C_{k,n}/C_{b,n})C^{(k)}_{k_1,\ldots,k_b}}{\Gamma(s-k+1) k!}
\\ \nonumber && \quad\qquad\qquad \times \sum_{j_1\neq \ldots \neq j_b}
\frac{\|\bep_{j_1}\|^{k_1}\cdots\|\bep_{j_b}\|^{k_b}(M^2\sigma^2r/n)^{(s-\ell)/2}}
{n(n-1)\cdots(n-b+1)n^{\ell-b}}
\cr &\le & 2C^*_{m,n}
C_0\max_{k_1+\ldots + k_b = \ell \le k\le m \atop k_a\ge 2\,\forall a; b\ge 0}
\sum_{j_1\neq \ldots \neq j_b}
\frac{\|\bep_{j_1}\|^{k_1}\cdots\|\bep_{j_b}\|^{k_b}(M^2\sigma^2r/n)^{(s-\ell)/2}}{n^{\ell}}
\ees
as in the beginning of the proof of Theorem \ref{th-risk-bd}.
Moreover, for $k_1,\ldots,k_b$ satisfying $k_a\ge 2$ and $k_1+\ldots + k_b = \ell$,
$\|\bw\|_{k_a}\le\|\bw\|_2$ and \eqref{M-bd} gives
\bes
\sum_{j_1\neq \dots \neq j_b}
\frac{\|\bep_{j_1}\|^{k_1}\cdots\|\bep_{j_b}\|^{k_b}}
{n^{\ell}}
= \prod_{a=1}^b \|\bw\|_{k_a}^{k_a}\le \|\bw\|_2^\ell \le \big(M^2\sigma^2r/n\big)^{\ell/2}
\ees
It follows that under \eqref{th-1-2}, in the event $\Omega_M$ and for $s=O(1)$
\bel{pf-CLT-1}
|\Rem_m| \le 2C^*_{m,n}C_{m,n}C_0(M^2\sigma^2r/n)^{s/2} \ll \sqrt{V_1/n}.
\eel

Next, we consider the $U$-variables $S_k$.
By \eqref{S_k}, $S_k$ are uncorrelated as they are completely degenerate $U$-variables of order $k$,
so that \eqref{var-S_k-inid}, \eqref{var-S_k-bd} and \eqref{th-1-2} yield
\bel{pf-CLT-2}
\quad && \EE\bigg[\bigg(\sum_{k=2}^m \frac{S_k}{k!}\bigg)^2\bigg]
\le \sum_{k=2}^m \frac{(C_{k,n}^2)C_0^2\sigma^{2k}r^{k-1}}{n^k k!}
\le C^*_{m,n}\max_{k\in [2,s]}\frac{\sigma^{2k}r^{k-1}}{n^k} \ll V_1/n
\eel
as the maximum over $k\in [2,s]$ is attained at the endpoints.
In view of the above bounds for $|\Rem_m|$ and $|\sum_{k=2}^m S_k/k!|$ and
the definition of $S_k$ in \eqref{S_k}, we find that
the linear term with $k=1$ dominates in the expansion \eqref{new-prop-1-1}.
Thus, \eqref{th-1-3} follows from the Lindeberg central limit theorem.

Without the condition $s = O(1)$, the factor $M^s$ in \eqref{pf-CLT-1} is unbounded but
the constant factor $C^*_{m,n}$ in \eqref{pf-CLT-1} can be improved by using \eqref{lm-1-1}
instead of \eqref{lm-1-2} as follows.
In the event $\Omega_M$, \eqref{M-bd} and the Hilbert-Schmidt smoothness condition yield
\bes
|\Rem_m|
&\le & 2\|f\|_{(s),{\rm HS},\btheta}
\sum_{k=0}^{m}
\frac{\|\bar{\bep}\|^{s-k}\langle \bar{\bep}^{(k)},\bar{\bep}^{(k)}\rangle_k^{1/2}}{\Gamma(s-k+1) k!}
\cr &\le & 2C_0 d^{m/2}
\sum_{k=0}^{m}
\frac{(M^2\sigma^2r/n)^{(s-k)/2}\langle \bar{\bep}^{(k)},\bar{\bep}^{(k)}\rangle_k^{1/2}}{\Gamma(s-k+1) k!}.
\ees
Because $\bar{\bep}^{(k)}$ are degenerate $U$-tensors in \eqref{new-eps-(k)},
\bes
\EE\big[ \langle \bar{\bep}^{(k)},\bar{\bep}^{(k)}\rangle_k\big]
= k! \sum_{1\le j_1 < \cdots < j_k\le n}
\frac{\prod_{a=1}^k \EE\big[\|\bep_{j_a}\|^2\big]}{n(n-1)\cdots(n-k+1)}
\le  k! C_{k,n}(\sigma^2r/n)^k.
\ees
It follows that
\bes
&& \EE\bigg[ d^{m/2}\sum_{k=0}^{m}
\frac{(M^2\sigma^2r/n)^{(s-k)/2}\langle \bar{\bep}^{(k)},\bar{\bep}^{(k)}\rangle_k^{1/2}}{\Gamma(s-k+1) k!}\bigg]
\cr &\le & d^{m/2} \sum_{k=0}^{m} \frac{(M^2\sigma^2r/n)^{s/2}\sqrt{m! C_{m,n}}}{(m-k)! k!}
\cr & = & d^{m/2} (M^2\sigma^2r/n)^{s/2}\sqrt{C_{m,n}} 2^m/\sqrt{m!}
\cr & = & M^{s-m} (\sigma^2r/n)^{s/2}\sqrt{C_{m,n}}
\sqrt{4^md^mM^{2m}/m!}
\cr & \le & (1+M)(\sigma^2r/n)^{s/2}\sqrt{C_{m,n}} e^{2dM^2}.
\ees
Thus, for $d=O(1)$ and in an event $\Omega^*_M$ with at least probability $1-3/M^2$,
\bes
|\Rem_m|
\le 2C_0M^2(1+M)(\sigma^2r/n)^{s/2} e^{1+2dM^2} \ll (V_1/n)^{1/2}
\ees
as $C_{m,n} \le e^{m^2/n}\le e$.  Moreover, instead of \eqref{pf-CLT-2} we have
\bes
\EE\bigg[\bigg(\sum_{k=2}^m \frac{S_k}{k!}\bigg)^2\bigg]
\le \sum_{k=2}^m \frac{(C_{k,n}^2)C_0^2d^k\sigma^{2k}r^{k-1}}{n^k k!}
\le e^{d+2}C_0^2\max_{k\in [2,s]}\frac{\sigma^{2k}r^{k-1}}{n^k} \ll V_1/n.
\ees
As both the remainder and $\sum_{k=2}^m S_k/k!$ are of smaller order than
$(V_1/n)^{1/2}$ under the first condition of \eqref{th-1-2} and $d=O(1)$, \eqref{th-1-3} still holds
under the Lindeberg condition.
\end{proof}

\begin{proof}[Proof of Theorem \ref{th-separable}]
We first prove \eqref{separable-U-bd} and \eqref{separable-Rem-bd}.
Let $C_{k,n} = n^k(n-k)!/n! \le e^{(k-1)k/n}$ as in Lemma \ref{lm-1}.
It follows respectively from the independence of $\bep_1,\ldots,\bep_n$
and the first line of \eqref{noise-cond} that
$\EE[\bar{\veps}_a^{(k)}\bar{\veps}_b^{(k')}]=0$ for $k \neq k'$ and
\bes
\EE\Big[\bar{\veps}_a^{(k)}\bar{\veps}_b^{(k)}\Big]
&=& \bigg(\frac{C_{k,n}k!}{n^k}\bigg)^2\sum_{1\le j_1 < \cdots < j_k\le n}
\prod_{\ell=1}^k \EE[\veps_{j_\ell,a}\veps_{j_\ell,b}]
\cr &=& \frac{C_{k,n}^2k!}{n^{2k}} \bigg(\sum_{1\le j_1 \neq \cdots \neq j_k\le n}
\prod_{\ell=1}^k \EE[\veps_{j_\ell,a}^2]\bigg)I_{\{a=b\}}
\cr &\le & C_{k,n}^2k!\bigg(\sum_{1\le j_1,\ldots, j_k\le n}
\prod_{\ell=1}^k \frac{\EE[\veps_{j_\ell,a}^2]}{n^2}\bigg)I_{\{a=b\}}
\cr & = & C_{k,n}^2k!\big(\EE\big[\bar{\veps}_a^2\big]\big)^kI_{\{a=b\}}.
\ees
This and the second line of \eqref{noise-cond} with $k=1$ give \eqref{separable-U-bd}:
\bes
\EE\bigg[\bigg(\sum_{a=1}^d \sum_{k=1}^m \frac{f_h^{(k)}(\theta_a)\bar{\veps}_a^{(k)}}{d(k!)}\bigg)^2\bigg]
\le \sum_{a=1}^d \sum_{k=1}^m C_{k,n}^2
\frac{|f_h^{(k)}(\theta_a)|^2\sigma_n^{2k}}{d^2k!}.
\ees

Let $\veps''_j$ be i.i.d. random variables independent of $\{\bep_1,\ldots,\bep_n\}$ with
$\P\{\veps''_j = 2\} = 1/3$ and $\P\{\veps''_j = -1\} = 2/3$.
Because $\EE[\veps''_i]=0$ and $\EE[(\veps''_i)^k]\ge 1$ for all integers $k\ge 2$, we have
\bes
\bigg|\EE\bigg[\prod_{\ell=1}^{2s}\veps_{j_\ell,a}\bigg]\bigg|
\le \EE\bigg[\prod_{\ell=1}^{2s}\big(\veps_{j_\ell}''\big|\veps_{j_\ell,a}\big|\big)\bigg]
\ees
for all $1\le j_1,\ldots,j_{2s}\le n$ and integers $s\ge 1$ regardless of multiplicity in the indices. Thus,
\bes
C_{k,n}^{-2}\EE\big[\big(\bar{\veps}_a^{s-k}\bar{\veps}_a^{(k)}\big)^2\big]
&=& n^{-2s}\EE\bigg[\bigg(\sum_{1\le j_1\neq\cdots\neq j_k\le n}\sum_{1\le j_{k+1},\ldots,j_s\le n}
\prod_{\ell =1}^s \veps_{j_\ell,a}\bigg)^2\bigg]
\cr &\le & n^{-2s}\EE\bigg[\sum_{1\le j_1,\ldots, j_{2s}\le n}
\prod_{\ell =1}^{2s} \big(\veps_{j_\ell}''\big|\veps_{j_\ell,a}\big|\big)\bigg]
\cr & = & \EE\bigg[\bigg(\sum_{j=1}^n \veps_j''\big|\veps_{j,a}\big|/n\bigg)^{2s}\bigg].
\ees
By symmetrization with the Rademacher variables $\veps_j'$ and moment comparison, we have
\bes
\EE\bigg[\bigg(\sum_{j=1}^n \veps_j''\big|\veps_{j,a}\big|/n\bigg)^{2s}\bigg]
\le 2^{2s}\EE\bigg[\bigg(\sum_{j=1}^n \veps_j'\veps_j''\veps_{j,a}/n\bigg)^{2s}\bigg]
\le 2^{4s} \EE\bigg[\bigg(\sum_{j=1}^n \veps_j'\veps_{j,a}/n\bigg)^{2s}\bigg].
\ees
It follows from the above two inequalities and the second line of \eqref{noise-cond} that
\bes
C_{k,n}^{-1}\Big(\EE\big[\big(\bar{\veps}_a^{s-k}\bar{\veps}_a^{(k)}\big)^2\big]\Big)^{1/2}
\le 2^{2s} \bigg(\EE\bigg[\bigg(\sum_{j=1}^n \veps_j'\veps_{j,a}/n\bigg)^{2s}\bigg]\bigg)^{1/2}
\le 2^{2s} \sqrt{\sigma_n^{2s}2^{s-1}s!}.
\ees
Thus, by \eqref{separable-rem},
\bes
\EE\bigg[\bigg( \frac{1}{d}\sum_{a=1}^d \Rem_{m,a}\bigg)^2\bigg]
&\le& \bigg(\frac{\|f_h\|_{(s)}}{d}\sum_{a=1}^d\sum_{k=0}^s
\frac{C_{m,n}2^{2s} \sqrt{\sigma_n^{2s}2^{s-1}s!}}{(s-k)!k!}\bigg)^2
\cr & = & C_{m,n}^2\|f_h\|_{(s)}^2 \sigma_n^{2s}2^{7s-1}/s!.
\ees
This gives \eqref{separable-Rem-bd}.

Inserting \eqref{separable-U-bd} and \eqref{separable-Rem-bd} into \eqref{separable-expansion},
we find that \eqref{separable-error-bd} holds:
\bes
&& \sqrt{\EE\big[|\widehat{f}-f(\btheta)|^2\big]}
\cr &\le&
\bigg\{\bigg|\frac{1}{d}\sum_{a=1}^d (f_0-f_h)(\theta_a)\bigg|^2
+ \sum_{a=1}^d \sum_{k=1}^m \frac{|f_h^{(k)}(\theta_a)|^2\sigma_n^{2k}}{(d^2/2)k!}\bigg\}^{1/2}
+ \frac{\|f_h\|_{(s)}(2^{7/2}\sigma_n)^s}{\sqrt{s!}}
\ees
for integers $s$ satisfying $s>2$ and $(s-1)s\le(n/2)\log 2$.

Under the smoothness condition \eqref{f-cond-alpha}, \eqref{separable-error-bd} yields
\eqref{th-separable-1} via
\bes
\kappa^{1/2}_{s,h,n,d}(\btheta)
- \kappa^{1/2}_{\alpha,0,n,d}(\btheta)
\le \bigg\{\sum_{k=1}^{s-1} \frac{\eta_\alpha^2(h)(\sigma_n/h)^{2k}}{(d/2)k!}\bigg\}^{1/2}
\le \eta_\alpha(h)\sqrt{d/2}e^{(\sigma/h)^2/2},
\ees
$\|f_h\|_{(s)}\le\eta_\alpha(h)/h^s$, $s!\ge e^{-s}s^{s+1/2}\sqrt{2\pi}$
and $|\bias_{\alpha,h}|\le \eta_\alpha(h)$.
\end{proof}

\begin{proof}[Proof of Theorem \ref{th-ell_1}]
In the following proof, we use the same $C_1$ to bound different quantities to simplify notation although this does not give the sharpest constant factors in the risk bounds.
Let $y = (x-t)/h$ so that $t = x-hy$. By \eqref{kernel-smoothing} and the condition $\int K(x)dx =1$,
\bel{pf-ell_1-1}
f_h(x) = \int h^{-1}K((x-t)/h)f_0(t)dt = \int K(y)f_0(x-hy)dy.
\eel
By the conditions $\int K(x)dx =1$,
$|f_0(x)-f_0(y)|\le|x-y|$ and $\int |xK(x)|dx \le C_1$,
\bes
\big|f_h(x)-f_0(x)\big|
=\bigg|\int K(y)\big(f_0(x-hy) - f_0(x)\big)dy\bigg|
\le h\int |yK(y)|dy\le C_1h.
\ees
Moreover, \eqref{pf-ell_1-1} gives
$f_h^{(1)}(x) = \int K_h(y)f_0^{(1)}(x-y)dy = \int_{-\infty}^x K_h(y)dy - \int_x^\infty K_h(y)dy$
and $f_h^{(2)}(x) = 2K_h(x)$, so that $\|f_h^{(1)}\|_{L_\infty} \le \|K_h\|_{L_1}=\|K\|_{L_1}\le C_1$ and
$\|f_h^{(k)}\|_{L_\infty}=2\|K_h^{(k-2)}\|_{L_\infty} = 2h^{1-k}\|K^{(k-2)}\|_{L_\infty}\le C_1 h^{1-k}$
for $k\ge 2$. Thus, \eqref{K-cond-ell_1} implies \eqref{f-cond-alpha} with $\alpha=1$
and $\eta_1(h)=C_1h$,
and the conclusion follows from \eqref{th-separable-2} of Theorem \ref{th-separable}.

Next, we verify \eqref{K-cond-ell_1} under \eqref{Q-cond-ell_1}. By the Fourier transform formula,
\bes
Q(\zeta) = (2\pi)^{-1/2}\int_{-\infty}^\infty  e^{-i \zeta x}K(x)dx,\
K(x) = (2\pi)^{-1/2}\int_{-\infty}^\infty  e^{i \zeta x}Q(\zeta)d\zeta.
\ees
It follows that $\int K(x)dx = \sqrt{2\pi}Q(0)=1$ by the first condition in \eqref{Q-cond-ell_1},
\bes
2\|K^{(k-2)}\|_{L_\infty}\le \sqrt{2/\pi}\int_{-1}^1|\zeta^{k-2}Q(\zeta)|d\zeta \le \sqrt{2/\pi}\|Q\|_{L_1}\le C_1
\ees
for $k\ge 2$ by the second and third conditions in \eqref{Q-cond-ell_1}, and  for $\ell\in \{0,1\}$
\bes
\int |x^\ell K(x)|dx &\le& \bigg(\int (1+x^2)^{-1}dx \int (1+x^2)|x^\ell K(x)|^2dx\bigg)^{1/2}
\cr &=& \pi^{1/2}\big(\|Q^{(\ell)}\|_{L_2}^2+\|Q^{(\ell+1)}\|_{L_2}^2\big)^{1/2}
\ees
is bounded by $C_1$ by the Plancherel identity and the fourth conditions in \eqref{Q-cond-ell_1}.
\end{proof}

\begin{proof}[Proof of Theorem \ref{th-ell_alpha}]
Because $K(x)$ is the Fourier inversion of $Q(\zeta)$,
\bes
K(x) = (2\pi)^{-1/2}\int_{-\infty}^\infty  e^{i \zeta x}Q(\zeta)d\zeta,\
Q(\zeta) = (2\pi)^{-1/2}\int_{-\infty}^\infty  e^{-i \zeta x}K(x)dx.
\ees
Thus, $\int K(x)dx = \sqrt{2\pi}Q(0)=1$ by the first condition in \eqref{Q-cond-ell_1},
and for $p\in [0,1]$
\bes
\int |x|^p|K(x)| dx
&\le& \bigg(\int (1+x^2)^{-1}dx \int (1+x^2)|x|^{2p}|K(x)|^2dx\bigg)^{1/2}
\cr &\le& (2\pi)^{1/2}\max_{k=0,2}\|Q^{(k)}\|_{L_2}
\ees
is bounded by $C_1$ by the Plancherel identity and the fourth conditions in \eqref{Q-cond-ell_1}.
It follows from the above properties of the kernel $K(x)$ and
the H\"older smoothness of $f_0$ that
\bes
\big|f_h(x)-f_0(x)\big|
=\bigg|\int K(y)\big(f_0(x-hy) - f_0(x)\big)dy\bigg|
\le h^p\int |y|^{p} |K(y)|dy\le C_1 h^{p}.
\ees

By the Fourier inversion formula,
\bes
K_h(x-y)
&=& h^{-1}(2\pi)^{-1/2}\int e^{i(x-y)\zeta/h}Q(\zeta)d\zeta
\cr &=& (2\pi)^{-1/2}\int e^{iy\zeta }\big(e^{-ix \zeta }Q(-h\zeta)\big)d\zeta
\ees
For $0<p<1$, the Fourier transformation of $f_0^{(1)}(x)=p \cdot \sgn(x)|x|^{p-1}$ and
\bes
(2\pi)^{-1/2}\int f_0^{(1)}(x) e^{-ix\zeta}dx
= C_{p} (-i\zeta)/|\zeta|^{1+p}
\ees
with $C_{p} = (2/\pi)^{1/2}\Gamma(p+1)\sin(\pi p/2)\in (0,\sqrt{2/\pi}]$.
Thus, by the Plancherel identity,
\bes
f_h^{(1)}(x)
&=& \int K_h(x-y)f_0^{(1)}(y)dy
\cr &=& \int_{-1/h}^{1/h} \Big(e^{-ix\zeta }Q(-h\zeta)\Big)
\Big(C_{p} (i\zeta)/|\zeta|^{1+p}\Big)d\zeta.
\ees
We note that $f^{(1)}$ is real-valued so that we only need to apply the complex conjugate to its Fourier transformation in the application of the Plancherel identity.
Moreover, $C_{p}\|Q\|_{L_1}\le(2/\pi)^{1/2}\|Q\|_{L_1}\le C_1$
by the third condition in \eqref{Q-cond-ell_1}.
Consequently,
\bes
\|f_h^{(k)}\|_{L_\infty}
\le \int_{-1/h}^{1/h} \big|Q(-h\zeta)\big|
C_{p} |\zeta|^{k-1-p}d\zeta
\le h^{p-k}C_{p} \|Q\|_{L_1} \le C_1 h^{p-k}.
\ees
for all $k\ge 1$ and $0<p<1$.
Hence, \eqref{Q-cond-ell_1} implies \eqref{f-cond-alpha} with $0<p<1$ and $\eta_{p}(h) = C_1h^{p}$.
The conclusion follows from \eqref{th-separable-2} of Theorem \ref{th-separable}.
\end{proof}

\bibliographystyle{amsalpha}
\bibliography{FunctionEst-submit.bib}
\end{document}

%% file: preambule-cun-hui.tex
\newcommand{\bel}{\begin{eqnarray}\label}
\newcommand{\eel}{\end{eqnarray}}
\newcommand{\bes}{\begin{eqnarray*}}
\newcommand{\ees}{\end{eqnarray*}}
\newcommand{\bei}{\begin{itemize}}
\newcommand{\eei}{\end{itemize}}
\newcommand{\beiftnt}{\begin{itemize}\footnotesize}

\def\benu{\begin{enumerate}}
\def\eenu{\end{enumerate}}

\def\real{{\mathbb{R}}}
\def\R{{\real}}
\def\E{{\mathbb{E}}}

\def\P{{\mathbb{P}}}

\def\complex{\mathop{{\rm I}\kern-.58em\hbox{\rm C}}\nolimits}

\def\Rem{\hbox{\rm Rem}}
\def\sgn{\hbox{\rm sgn}}

\DeclareMathOperator{\trace}{trace}

\def\Var{\hbox{\rm Var}}

\DeclareMathOperator{\supp}{supp}

\def\mathbold{\boldsymbol} %\def\mathbold{\mathbf}
\def\calB{{\cal B}}

\def\bfe{\mathbold{e}}

\def\fhat{\widehat{f}}

\def\bh{\mathbold{h}}

\def\calH{{\cal H}}\def\scrH{{\mathscr H}}

\def\bI{\mathbold{I}}
\def\calI{{\cal I}}

\def\calN{{\cal N}}

\def\calT{{\cal T}}

\def\bu{\mathbold{u}}

\def\bv{\mathbold{v}}

\def\bw{\mathbold{w}}

\def\bx{\mathbold{x}}

\def\calX{{\cal X}}

\def\by{\mathbold{y}}

\def\ep{\varepsilon}\def\veps{\varepsilon}
\def\bep{ {\mathbold{\ep} }}

\def\bzeta{\mathbold{\zeta}}

\def\btheta{\mathbold{\theta}}

\def\lam{\lambda}

\def\bSigma{\mathbold{\Sigma}}